\documentclass{article}

\usepackage[utf8]{inputenc}
\usepackage[T1]{fontenc}
\usepackage{amssymb}
\usepackage{amsmath}
\usepackage{amsthm}
\usepackage{mdframed}
\usepackage{enumerate}
\usepackage{verbatim}
\usepackage[a4paper,top=3cm,bottom=5.5cm]{geometry}
\usepackage[colorlinks=true,linkcolor=black,bookmarks]{hyperref}
\usepackage{aliascnt}
\usepackage{lmodern}
\usepackage{mdwlist}
\usepackage[activate]{pdfcprot}
\usepackage{graphicx}
\usepackage{slashed}
\usepackage{mathrsfs}
\usepackage{mathtools}
\usepackage{mathabx}
\usepackage{mathdots}
\usepackage{dsfont}
\usepackage{extpfeil}
\usepackage{tabularx}	
\usepackage{multirow}
\usepackage[capitalise]{cleveref}
\usepackage{palatino}
\usepackage{hyperref}
\usepackage{subcaption}
\usepackage{color}
\usepackage{placeins}
\usepackage{afterpage}
\usepackage[ruled,vlined]{algorithm2e}
\usepackage[super]{nth}

\usepackage[all]{xy}
\usepackage{float}

\newtheorem{proposition}{Proposition}
\numberwithin{proposition}{section}
\newtheorem*{proposition*}{Proposition}
\newtheorem{theorem}{Theorem}
\numberwithin{theorem}{section}
\newtheorem*{theorem*}{Theorem}
\newtheorem{lemma}{Lemma}
\newtheorem*{lemma*}{Lemma}
\numberwithin{lemma}{section}
\newtheorem{corollary}{Corollary}
\newtheorem*{corollary*}{Corollary}
\numberwithin{corollary}{section}
\theoremstyle{remark}
\newtheorem{remark}{Remark}
\newtheorem*{remark*}{Remark}
\numberwithin{remark}{section}
\theoremstyle{definition}
\newtheorem{definition}{Definition}
\numberwithin{definition}{section}
\newtheorem*{definition*}{Definition}

\newtheorem*{example*}{Example}
\numberwithin{example}{section}

\numberwithin{counterexample}{section}

\numberwithin{assumption}{section}
\newenvironment{preamble}{}{}

\newcommand{\supp}{\operatorname{supp}}

\newcommand{\diag}{\operatorname{diag}}

\newcommand{\dist}{\operatorname{dist}}

\renewcommand{\setminus}{\smallsetminus}

\def\clap#1{\hbox to 0pt{\hss#1\hss}}

\newcommand{\tostar}{\overset{*}{\lower0.5em\hbox{$\smash{\scriptscriptstyle\rightharpoonup}$}}}

\DeclareMathOperator*{\prox}{Prox}
\DeclareMathOperator*{\sign}{sign}
\DeclareMathOperator*{\argmin}{arg\,min}

\newcolumntype{M}[1]{>{\centering\arraybackslash}m{#1}}

\usepackage[numbers]{natbib} 
\usepackage[table]{xcolor}

\begin{document}
\newgeometry{left=24mm,right=34mm, top=24mm, bottom=34mm}

\begin{preamble}

	\title{Optimal Control of the Linear Wave Equation by Time-Depending BV-Controls: A Semi-Smooth Newton Approach}

\author{S. Engel  
\footnote{
Institute for Mathematics and Scientific Computing, Karl-Franzens-Universität, Heinrichstr. 36, 8010 Graz, Austria
(\url{sebastian.engel@uni-graz.at})}
$\,$ and K. Kunisch
\footnote{
Institute for Mathematics and Scientific Computing, Karl-Franzens-Universität, Heinrichstr. 36, 8010 Graz, Austria, and Radon Institute, Austrian Academy of Sciences, (\url{karl.kunisch@uni-graz.at
})}
}
\date{August 2018}
\maketitle

\end{preamble}

\begin{abstract}
An optimal control problem for the linear wave equation with control cost chosen as the BV semi-norm in time is analyzed. This formulation  enhances piecewise constant optimal controls and penalizes the number of jumps. Existence of optimal solutions  and necessary optimality conditions are derived. With numerical realisation in mind, the regularization by $H^1$ functionals is investigated, and the asymptotic behavior as this regularization tends to zero is analyzed. For the $H^1-$regularized problems the 
semi-smooth Newton algorithm can be used to solve the first order optimality conditions with super-linear convergence rate. Examples are constructed which show that the distributional derivative of an optimal control can be a mix of absolutely continuous measures with respect to the Lebesgue measure, a countable linear combination of Dirac measures, and Cantor measures. Numerical results illustrate and support the analytical results.
\end{abstract}
\textcolor{white}{text}\\ \\
\textbf{AMS subject classifications.} 26A45, 35L05, 35L10, 49J15, 49J20, 49J52, 49K20.
\\ \\

\textbf{Key words.}	Wave equation, optimal control problems, sparsity, non-smooth analysis, functions of bounded variation, semi-smooth Newton method.



\section{Introduction}
We investigate the following optimal control problem for the wave equation with homogeneous Dirichlet boundary condition:
\begin{align*}
(P)\left\lbrace \begin{matrix}
& \min\limits_{u\in BV(0,T)^m} \frac{1}{2}\| y_u-y_d \|_{L^2(\Omega_T)}^2 + \sum\limits_{j=1}^m \alpha_j \|D_tu_j\|_{M(I)}=:J(y,u)\\ \\
\text{subject to } & (\mathcal{W}) \left\lbrace\begin{matrix}
\partial_{tt} y_u - \bigtriangleup y_u =  \sum\limits_{j=1}^m u_jg_j  &\text{ in }&(0,T) \times \Omega \\
y_u = 0  &\text{ on }&(0,T)\times \partial \Omega \\[1.1ex]
(y_u,\partial_t y_u) =   (y_0 , y_1)  &\text{ in }& \lbrace 0 \rbrace \times \Omega,
\end{matrix}\right.
\end{matrix}\right.
\end{align*}
where $\Omega \subset \mathbb{R}^n$, with  $n \in  \{1,2,3\}$,  is a bounded open domain with Lipschitz boundary $\Gamma:=\partial \Omega$, $T\in (0,\infty)$, and $y_d\in L^2((0,T)\times\Omega)$.  
The temporally dependent controls $u$ are chosen as $u=(u_1,...,u_m)\in BV(0,T)^m$, and $BV(0,T)^m$  is endowed with the norm $\|u\|_{BV(I)^m}=  \sum\limits_{j=1}^m  \|u_j\|_{L^1(I)} + \|D_tu_j\|_{M(I)}$, where  $I:=(0,T)$. 
Here  ${M(I)}$  denotes the space of Borel measures, endowed with the total variation norm  $\|\cdot \|_{M(I)}$ . Further let $(g_j)_j^m \subset L^\infty(\Omega)\setminus \lbrace 0 \rbrace$ with pairwise disjoint supports $w_j:=\supp(g_j)$, and $\alpha_j>0$. The initial data are chosen as  $(y_0,y_1) \in H^1_0(\Omega) \times L^2(\Omega)$ and we abbreviate $H:=L^2(\Omega)$, $V:=H^1_0(\Omega)$, with $V^*:=H^{-1}(\Omega)$. Finally we set $\Omega_T:=(0,T)\times \Omega$.

In problem $(P)$, we focus our attention on sparse optimal controls in the sense that they are piecewise constant. In particular, using the total variation of a BV-function in the cost functional $J$, enhances sparsity in the derivative of the optimal control.
For a piecewise constant optimal control of $(P)$ the jumps are located in the position of these Dirac measures, see for example \cite{[KCK]}. 
This type of sparsity property is reflected in the necessary and
sufficient first-order optimality condition.
As far as the authors know, the $L^1$-norm is one of the first discussed sparsity enhancing cost terms in the context of partial differential equations. A detailed discussion on the history of sparsity in optimal control of partial differential equations can be found in e.g. \cite{CaSEMA}. Furthermore,
sparsity results for optimal control problems with linear partial differential equations are considered in several works. References were specified for example in \cite{CKNavier} where the authors emphasize the papers \cite{CKParabST}, \cite{CaZu}, \cite{CCKApprox}, \cite{CCKParab}, \cite{VexZu}, \cite{CLKDual}, \cite{CLKMeas}, \cite{KPVMeas}, and \cite{KTVundam}. In image reconstruction, $BV$-functions are well investigated but modeling aspects are different compared to optimal control with partial differential equations. In mathematical image analysis the use of BV-functionals is motivated by their ability to preserve natural edges and corners in the image. An introduction to image reconstruction aspects can be found in \cite{CCNCP}.

For the purpose of numerical realization we rely on regularized problems by using the $H^1$ semi-norm. This enables us to approximate the $BV$ optimal control of $(P)$ by the $H^1$ controls in the strict-$BV$ sense. The main purpose of this regularization is to use the semi-smooth Newton algorithm for which we present super-linear convergence results. In particular, one is able to show that the regularized problem permits a point-wise formula for the derivative of the $H^1$ controls. This property is used for the well-posedness result of the Newton algorithm.

Let us briefly outline the following sections. In section \ref{WaveBVSection} we gather the necessary prerequisites on the wave equation and on one-dimensional $BV$-functions which will be needed later on in this paper. Section \ref{AnalysofP} is dedicated to the analysis of the optimal control problem and sparsity properties of the optimal controls.
Section \ref{regularizeSEct} is devoted to the regularized problem $(P^1_\gamma)$, the corresponding convergence results for the optimal controls of $(P^1_\gamma)$ as $\gamma\rightarrow 0$, and the first-order optimality conditions of $(P^1_\gamma)$.

Furthermore,  the  semi-smooth Newton algorithm and its  super-linear convergence are presented. The algorithm is embedded into a path following algorithm to approximate the original unregularized problem. In section \ref{numericsection}, we construct test cases for problem $(P)$ in such a manner that exact analytic solutions for $(P)$ can be found. The construction steps can be used to build all types of distributional derivatives for the optimal controls $D_tu_j$. This means that $D_tu_j$ can be a mix of absolutely continuous measures with respect to the Lebesgue measure, a countable linear combination of Dirac measures, and Cantor measures, see for example \cite[p. 184]{[AmFuPa]}. The first numerical example refers to an optimal control that has finitely many jumps. In the second example, we construct an optimal control which can be characterized as a Cantor function.
In the last section \ref{remarkSect} we remark that our results can be extended  to  several other linear second-order hyperbolic equations.

\section{The Wave Equation and $BV$ Functions in Time}
\label{WaveBVSection}
\subsection{Preliminaries on the Wave Equation}
Since in this work non-smooth data are used for the wave equation, we directly introduce the weak solution of the wave equation (see e.g. \cite{[Zlot]}). In particular $y_u$ is understood as the weak solution of the wave equation $(\mathcal{W})$ in problem $(P)$. Furthermore, we present in this section standard regularity results, and an energy estimation for the weak solution of the wave equation.
\begin{definition}(\cite[Chap.IV, Sec.4]{[La68]})\label{weakDef} We call $y\in C([0,T];V)$ with $\partial_t y \in C([0,T];H)$ a weak solution of $(\mathcal{W})$ with forcing $f\in L^1(0,T;H)$, displacement $y_0\in V$, and velocity $y_1\in H$ if $y\vert_{t=0}=y_0$
	\begin{align}
		\begin{matrix}
		-\int\limits_I\left\langle\partial_t  y ,\partial_t \eta \right\rangle_{H}dt+\left\langle\bigtriangledown y ,\bigtriangledown\eta \right\rangle_{L^2(\Omega_T)}=\left\langle y_1,\eta\vert_{t=0} \right\rangle_{H}+\int\limits_I\left\langle f,\eta \right\rangle_{H}dt
		\end{matrix}
	\end{align}
	for every $\eta\in L^1(I;V)$ such that $\partial_t \eta \in L^1(I;H)$, $\eta\vert_{t=T}=0$. \label{weakSoluti}
\end{definition}

\begin{theorem}[\cite{[Zlot]}]\label{RegTheoremWave}\label{energy}
For each $(f,y_0,y_1)\in L^1(0,T;H)\times V \times H$ there exists a unique weak solution $y=y(f,y_0,y_1)\in C([0,T];V)\cap C^1([0,T];H)$ of $(\mathcal{W})$.

The mapping $(f,y_0,y_1) \mapsto y(f,y_0,y_1)$ is linear and continuous from $L^1(0,T;H)\times V \times H$ into $C([0,T];V)\cap C^1([0,T];H)$.


In particular, there exists a constant $c>0$ such that for all $(f,y_0,y_1) \in L^1(I;H) \times V \times H$ the unique weak solution $y=y(f,y_0,y_1)$ satisfies
\begin{align}\label{EnergyEsti}
\begin{matrix}
\|y\|_{\mathcal{C}(\overline{I};V)} + \|\partial_t y\|_{\mathcal{C}(\overline{I};H)} \leq c\left( \|f\|_{L^1(I;H)} + \|y_0\|_{V}  +\|y_1\|_{L^2}\right)
\end{matrix}
\end{align}

\end{theorem}

For the proof we refer to \cite[Proposition 1.1]{[Zlot]}.
\begin{definition}
	Let us define the following continuous linear operators
	\begin{align*}
	\begin{matrix}
	\begin{matrix}
	L:& L^2(\Omega_T) & \rightarrow  & L^2(\Omega_T)\\
	& f & \mapsto & y(f,0,0)
	\end{matrix} & \text{ and } &
	\begin{matrix}
	Q:& V\times H & \rightarrow  & L^2(\Omega_T)\\
	&(y_0,y_1) & \mapsto & y(0,y_0,y_1)
	\end{matrix}
	\end{matrix}
	\end{align*}

	Furthermore, we define the continuous affine linear solution operator $\tilde{S}$:
	\begin{align}\label{solution_op_def123}
	\begin{matrix}
	\begin{matrix}
	\tilde{S}:L^2(I)^m &\rightarrow & L^2(\Omega_T)\\
	u & \mapsto & L(ug)+Q(y_0,y_1)
	\end{matrix} & \text{with $ug=\sum\limits_{j=1}^m u_jg_j$.}
	\end{matrix}
	\end{align}
\end{definition}

\begin{lemma}\label{adjointLemma}
	The action of the adjoint operator $L^*$ is given by
	\begin{align}\label{adj_funct_wave}
	\begin{matrix}
	L^*:L^2(\Omega_T)\rightarrow L^2(\Omega_T)\\ \\
	w \mapsto p
	\end{matrix}
	\end{align}
\end{lemma}
with $p(t,x)=y(w(T-\cdot),0,0)(T-t,x)$.

	\subsection{Preliminaries on $BV$ Functions in Time}
	Concerning BV-functions in one scalar variable we refer to \cite{[AmFuPa]}. In this section we only recall a few facts which we frequently refer to:\\ \\
	A sequence $(u_k) \subset BV(I)$ is said to converge weakly* in $BV(I)$ to $u$ if $(u_k)$ converges to $u$ in $L^1(I)$, and the measures $(Du_k)$ converge weakly* in the measure space $M(I)$ to $Du$, i.e. $\lim\limits_{k\rightarrow \infty} \int\limits_{0}^T \varphi Du_k = \int\limits_{0}^T \varphi Du $ for every $\varphi \in C_0(I)$.

	For all bounded sequences $(u_k)_k\subset BV(I)$ there exists a weakly* convergent sub-sequence with limit $u\in BV(I)$.

	 A weakly*-converging sequence $(u_k)_k$ in $BV(I)$ with limit $u$ is also strongly converging in $L^p(I)$ for $1\leq p<\infty$ to $u$.

	\label{lsc_of_TV}
	The total variation functional $\|D_t \cdot\|_{M(I)}:L^1(I)\rightarrow \overline{\mathbb{R}}:=\mathbb{R}\cup \lbrace \infty \rbrace$ is convex and lower semi-continuous with respect to $L^1(I)-$convergence \cite[Proposition 3.6]{[AmFuPa]}.

	A sequence $(u_k) \subset BV(I)$ is said to converge strictly in $BV(I)$ to $u$ if $(u_k)$ converges strongly in $L^1(I)$ and $\|D_tu_k\|_{M(I)} \xrightarrow{k\rightarrow \infty}\|D_tu\|_{M(I)}$. Strictly converging sequences in $BV(I)$ are also weakly* converging in $BV(I)$, see \cite[p. 126]{[AmFuPa]}.

	The following BV-Poincaré inequality holds:
	\begin{lemma}[\cite{[AmFuPa]}, page 152] There exists $c>0$ such that for all $u\in BV(I)$ and $1\leq\sigma\leq\infty$ we have
		$\|u - a\|_{L^\sigma(I)}\leq c \|D_tu\|_{M(I)}$, with  $a:=\frac{1}{T}\int\limits_I u dx$.
		
	\end{lemma}
	\begin{lemma} \label{isoBV}
		For each $m\in \mathbb{N}_{>0}$ the map $u  \rightarrow  (D_tu,u(0))$ is an isomorphism from $BV(I)^m$ to $M(I)^m\times \mathbb{R}^m$ with inverse $(v,c)\rightarrow \int\limits_{[0,t)} dv+c$. A similar result holds for $H^1(I)^m$ and $L^2(I)^m\times\mathbb{R}^m$.
	\end{lemma}
		
	\section{Analysis of the Optimal Control Problem (P)}\label{AnalysofP}
	In the following we show the existence of a unique solution of $(P)$. Furthermore, we will introduce a problem $(\tilde{P})$ which is equivalent to $(P)$, for which the first-order optimality conditions are derived. These optimality conditions will be used to present sparsity results for the optimal control of $(P)$.
	\begin{theorem}\label{ExistenceSoluti}
		Problem $(P)$ has a unique solution $\overline{u}\in BV(I)^m$.
	\end{theorem}
	\begin{proof} Utilizing the fact that the forward mapping is continuous from $L^2(I)^m$ to $L^2(\Omega_T)$, the proof can be carried out along the lines of \cite[Theorem 3.1]{[KCK]}.
	\end{proof}
	\subsection{Equivalent Problem $(\tilde{P})$}
	
	Consider the following linear and continuous operator:
	\begin{align}\label{B_red_to_L_2}
	\begin{matrix}
	\begin{matrix}
	B: & M(I)^m \times \mathbb{R}^m & \rightarrow & L^2(\Omega_T)\\
	& (v,c) & \mapsto & \sum\limits_{j=1}^m \left( \int\limits\limits_{[0,t)} dv_j + c_j \right)g_j
	\end{matrix}
	\end{matrix}
	\end{align}

	Using the identification of $BV(I)$ with $M(I)\times \mathbb{R}$, see Lemma \ref{isoBV}, and the fact that $BV(I)$ embeds into $L^2(I)$ we can rewrite $(P)$ as the equivalent problem:
	\begin{align*}
	(\tilde{P})\left\lbrace\begin{matrix}
	& \min\limits_{
		\begin{matrix}
		v\in M(I)^m \\
		c\in \mathbb{R}^m
		\end{matrix}} \frac{1}{2}\| \tilde{S}(v,c)-y_d \|_{L^2(\Omega_T)}^2 + \sum\limits_{j=1}^m \alpha_j \|v_j\|_{M(I)}=:\tilde{J}(v,c),
	\end{matrix}\right.
	\end{align*}
	where we have to modify the control to state operator $\tilde{S}$ to
	\begin{align*}
	\begin{matrix}
	S: & M(I)^m \times \mathbb{R}^m & \longrightarrow & L^2(\Omega_T)\\
	& (v,c) & \mapsto & L(B(v,c))+Q(y_0,y_1)
	\end{matrix}
	\end{align*}

	\subsection{First-Order Optimality Condition for $(\tilde{P})$}
In this section the necessary and sufficient first-order optimality conditions for $(\tilde{P})$ are presented. Furthermore, we show sparsity results for the optimal control of $(\tilde{P})$ respectively $(P)$. Let us begin with the following theorem:

\begin{theorem}\label{FirstOrder}
	The element $(\overrightarrow{v},\overrightarrow{c})\in M(I)^m\times\mathbb{R}^m$, is an optimal control for $(\tilde{P})$ if
	\begin{align}\label{p_1_i_function}
	\begin{matrix}
	-\begin{pmatrix}
	p_1(s)\\ p_1(0)
	\end{pmatrix}:=- \begin{pmatrix}
	\int\limits_{s}^T\int\limits_\Omega L^* \left( S\begin{pmatrix}
	\overrightarrow{v},\overrightarrow{c}
	\end{pmatrix}-y_d \right) \overrightarrow{g}dxdt \\ \\
	\int\limits_{\Omega_T} L^* \left( S\begin{pmatrix}
	\overrightarrow{v},\overrightarrow{c}
	\end{pmatrix}-y_d \right) \overrightarrow{g}dxdt
	\end{pmatrix}\in \begin{pmatrix}
	\left( \alpha_i \partial \| \overrightarrow{v}_i \|_{M(I)}\right)_{i=1}^m \\
	0_{\mathbb{R}^m}
	\end{pmatrix}
	\end{matrix}
	\end{align}
	where $p_1\in C^2([0,T])^m$. This first-order optimality condition is equivalent to:
For all $i=1,...,m$ and $v\in M(I)$ it holds that
	\begin{align*}
		\begin{matrix}
		\langle v-\overrightarrow{v}_i,-p_{1,i}\rangle_{M(I),C_0(I)} \leqslant \alpha_i \|v\|_{M(I)} - \alpha_i \|\overrightarrow{v}_i\|_{M(I)} \text{ and } p_1(0)=0_{\mathbb{R}^m}.
		\end{matrix}
	\end{align*}
\end{theorem}
\begin{proof}
	The proof can be found in the appendix.
\end{proof}
\begin{lemma}\label{support_measure_cor}
	Let $(\overrightarrow{v},\overrightarrow{c})\in M(I)^m\times\mathbb{R}^m$ be an optimal control for $(\tilde{P})$. Then we have for all $i=1,\cdots,m$ and $p_{1}=(p_{1,i})_{i=1}^m$ given in (\ref{p_1_i_function}):
	\begin{itemize}
		\item[a)]$\|p_{1,i}\|_{C_{0}(I)}\leq \alpha_i$
		\item[b)]$\int\limits_I -\frac{p_{1,i}}{\alpha_i}dv_i=\int\limits_I d\vert \overrightarrow{v}_i\vert=\|\overrightarrow{v}_i\|_{M(I)}$
		\item[c)]$\supp(\overrightarrow{v}_i^{\pm})\subseteq \lbrace t \in I \vert p_{1,i}(t)=\mp \alpha_i\rbrace$
	\end{itemize}
\end{lemma}
The proof is analogous to the one of \cite[Proposition 2.4]{[KCK]}.
The following corollary which is similar to a result in \cite{[KCK]} exhibits an important structural property of the solution $\overline{u}_{\alpha,j}$ as a function of $\alpha_j$.
\begin{corollary}
There exists $M_j>0$ such that the j-th component $\overline{u}_{\alpha,j}$ of the optimal control $\overline{u}_{\alpha}$ of $(P)$ is constant in $BV(I)^m$ for all $\alpha_j>M_j$.
\end{corollary}
\begin{proof}
	 Let $y^0$, $y_\alpha$ be the solutions of the state equation associated to the controls $u=0$ respectively $\overline{u}_\alpha$. Furthermore, let us define $\overline{p}_{\alpha}:=L(\overline{y}_{\alpha}-y_d)$. From the optimality of $\overline{u}_\alpha$ we get
	\begin{align*}
	\begin{matrix}
	\frac{1}{2} \|\overline{y}_\alpha - y_d\|_{L^2(\Omega_T)}^2 \leq J(\overline{u}_\alpha)\leq J(0) = \frac{1}{2}\| y^0 - y_d \|_{L^2(\Omega_T)}^2.
	\end{matrix}
	\end{align*}
	This implies that
$
	\|\overline{y}_\alpha - y_d\|_{L^2(\Omega_T)} \leq \|y^0 - y_d\|_{L^2(\Omega_T)}.
$
	From the adjoint state equation we obtain
	\begin{align*}
	\begin{matrix}
	\|\overline{p}_\alpha\|_{L^\infty(I;H)}
	= c_1 \|\overline{p}_\alpha\|_{C(\overline{I};V)} \leq c_1c_2 \|\overline{y}_\alpha -y_d\|_{L^1(I;H)}\\ \\
	\leq c_1c_2c_3 \|\overline{y}_\alpha -y_d\|_{L^2(\Omega_T)} \leq c_1c_2c_3 \|y^0 -y_d\|_{L^2(\Omega_T)}.
	\end{matrix}
	\end{align*}
	The constant $c_1$ is defined with respect to the embedding $L^\infty(I;V) \hookrightarrow L^\infty(I;H)$, $c_2$ is depending on the embedding constant in (\ref{EnergyEsti}), and $c_3$ is the embedding constant of $L^2(\Omega_T)=L^2(I;H)\hookrightarrow L^1(I;H)$. From the adjoint $p_1$, and the above estimation we get for all $t\in [0,T]$
	\begin{align*}
	\begin{matrix}
	|p_{1,j}(t)| \leq T \| \overline{p}_\alpha \|_{L^\infty(I;H)} \|g_j\|_{L^2(w_j)}\leq Tc_1c_2c_3 \|g_j\|_{L^2(w_j)}\|y^0 -y_d\|_{L^2(\Omega_T)} =: M_j
	\end{matrix}
	\end{align*}
	 where the first inequality follows from
	\begin{align*}
	\begin{matrix}
	|p_{1,j}(t)|  \leq \int\limits_t^T \|\overline{p}_\alpha \|_{L^2(w_j)} \|g_j\|_{L^2(w_j)} ds
	\leq T \| \overline{p}_\alpha \|_{L^\infty(I;H)} \|g_j\|_{L^2(w_j)}.
	\end{matrix}
	\end{align*}
	The support relation in Lemma \ref{support_measure_cor} now implies that $D_t\overline{u}_{\alpha,j}\equiv 0$ if $\alpha_j > M_j$.
\end{proof}
\begin{corollary}
Let $\overline{u}\in BV(I)^m$ be the optimal control of $(P)$. Assume for some $i\in\lbrace 1,\cdots,m\rbrace$ that the measures $D_t\overline{u}_i^+$ and $D_t\overline{u}_i^-$ are not trivial. Then we have
\begin{align*}
\begin{matrix}
	\dist(\supp(D_t\overline{u}_i^+),\supp(D_t\overline{u}_i^-)):=\min\limits_{x^\pm\in \supp(D_t\overline{u}_i^\pm)}\vert x^+-x^- \vert>0.
\end{matrix}
\end{align*}
\end{corollary}
\begin{proof}
W.l.o.g. let us consider $m=1$. Assume that $\dist(\supp(D_t\overline{u}^+),\supp(D_t\overline{u}^-))=0$. Then there exists a sequence $(t_n)_n\in \supp(D_t\overline{u}^+) \subset I$ such that $p_1(t_n)=-\alpha$ and $\dist(\lbrace t_n \rbrace,\supp(D_t\overline{u}^-))\rightarrow 0$. Hence, there exists a sub-sequence $(t_{t_k})_k$ which converges to some $\tilde{t}$ with $dist(\lbrace \tilde{t}\rbrace,\supp(D_t\overline{u}^-))=0$. Furthermore, there exists a sequence $(\tau_n)_n\in \supp(D_t\overline{u}^-)\subset I$ such that $p_1(\tau_n)=\alpha$ and $\tau_n\rightarrow  \tilde{t}$. By the continuity of $p_1$ we have $-\alpha=\lim\lim\limits_{k\rightarrow\infty}p(t_{n_k})=\lim\limits_{n\rightarrow \infty}p(\tau_n)=\alpha$ which is a contradiction to $\alpha>0$.
\end{proof}

	\begin{remark}\label{supp_opti_control_2}
	If the set of points in which $p_{1,i}(t) \in \lbrace \pm \alpha_i \rbrace$, is finite, we have by Lemma \ref{support_measure_cor} c) that $D_t\overline{u}_i$ is a combination of Dirac measures centered at those points (not necessarily in all of these points). In particular, we obtain that the optimal control $\overline{u}_j$ of $(P)$ is piecewise constant in $[0,T]$ with jumps in $\supp(D_t\overline{u}_i)$. This remark can also be found in \cite[Remark 3.5]{[KCK]}. Later we will construct an analytically exactly solvable example for our problem $(P)$, which allows us tho show that the derivatives of the optimal controls can either be of Cantor or Dirac kind or alternatively absolutely continuous with respect to the Lebesgue measure. In particular, the derivatives of the optimal controls need not to be sparse. For further information about these characterizations of measures, see for example \cite{[AmFuPa]} on page 184.
	\end{remark}

 \section{Regularization}
 \label{regularizeSEct}
 For numerical realization we aim at applying a semi-smooth Newton method. For this purpose we regularize problem $(P)$. We then analyze the asymptotic behavior of the optimal controls of the regularized problem as well as the first-order optimality condition of the regularized problem. Finally, we will present convergence results for the semi-smooth Newton algorithm.

 In the following, let us consider the regularized optimal control problem:
   \begin{align*}
  (P_\gamma^1)\left\lbrace \begin{matrix} \min\limits_{u\in H^1(I)^m} & \begin{bmatrix}
   \frac{1}{2}\| y_u-y_d \|_{L^2(\Omega_T)}^2 + \sum\limits_{j=1}^m \alpha_j \|\partial_t u_j\|_{L^1(I)}\\ \\ + \frac{\gamma}{2} \sum\limits_{j=1}^m \|\partial_t u_j\|_{L^2(I)}^2+\frac{\kappa(\gamma)}{2} \|u(0)\|_{\mathbb{R}^m}^2
  \end{bmatrix}
  =:J^1_\gamma(y,u)
  \text{ subject to }(\mathcal{W})
  \end{matrix}\right.
  \end{align*}
  with $\gamma >0$, $\kappa(\gamma)=c_\kappa\cdot \tilde{f}(\gamma)$, $c_\kappa\geq 0$, monotonously increasing, $\tilde{f}\in C^1([0,\infty))$, $\tilde{f}(0)=0$, and $\supp(\tilde{f})=[0,\infty)$. Note that for each $u\in H^1(I)$ the value $u(0)$ is well defined, because $H^1(I)$ embeds continuously into $C(\overline{I})$.
 %
 The total variation cost term in $(P)$ can be identified with the cost term $\sum\limits_{j=1}^m \alpha_j \|\partial_t \cdot\|_{L^1(I)}$ for $H^1(I)^m$ functions in $(P^1_\gamma)$ since now $u\in H^1(\Omega)$. The symbol $\partial_t$ represents the weak derivative.

  \subsection{Asymptotic behavior as $\gamma \rightarrow 0^+$}
  In this section we show that the unique solution of $(P)$ can be approximated by the unique solutions of the problems $(P^1_\gamma)$ as $\gamma \rightarrow 0$.

  In terms of the reduced costs $J$, problem $(P)$ can be expressed as
  \begin{align*}
  (P)\left\lbrace \begin{matrix}
  & \min\limits_{u\in BV(I)^m} \frac{1}{2}\| \tilde{S}(u)-y_d \|_{L^2(\Omega_T)}^2 + \sum\limits_{j=1}^m \alpha_j \|D_tu_j\|_{M(I)}=:J(u)
  \end{matrix}\right. .
  \end{align*}
  Analogously, we have
 \begin{align*}
 (P_\gamma^1)\left\lbrace \begin{matrix}
 \min\limits_{u\in H^1(I)^m}
 \frac{1}{2}\| \tilde{S}(u)-y_d \|_{L^2(\Omega_T)}^2 + \sum\limits_{j=1}^m \alpha_j \|\partial_t u_j\|_{L^1(I)} + \frac{\gamma}{2} \sum\limits_{j=1}^m \|\partial_t u_j\|_{L^2(I)}^2+ \frac{\kappa(\gamma)}{2} \|u(0)\|_{\mathbb{R}^m}^2=:J^1_{\gamma}(u)
 \end{matrix}\right.
 \end{align*}
 The following result follows with standard techniques.
  \begin{theorem}
  For every $\gamma>0$ problem $(P^1_\gamma)$ has a unique solution $\overline{u}_\gamma\in H^1(I)^m$.
  \end{theorem}
 Let us denote the unique optimal controls of $(P)$ and $(P^1_\gamma)$ by $\overline{u}$ and $\overline{u}_\gamma$.
 To argue the BV-weak* and strict convergence of $\overline{u}_\gamma$ to $\overline{u}$ we use concepts from \cite{[Piep]} and \cite{[KCK]}.
  \begin{definition}
  The value function is defined as
  \begin{align*}
   \begin{matrix}
   \mathfrak{V}:[0,\infty)& \rightarrow & \mathbb{R}, &
   \gamma & \mapsto & \mathfrak{V}(\gamma):=J^1_\gamma(\overline{u}_\gamma),
   \text{ where } \overline{u}_\gamma=\argmin\limits_{u\in H^1(I)^m}J^1_\gamma(u).
      \end{matrix}
   \end{align*}
  \end{definition}
  \begin{lemma}\label{lemma_valuefct}
  The value function $\mathfrak{V}$ maps $[0,\infty)$ into $[J(\overline{u}),\infty)$
  and $\mathfrak{V}(0):=J(\overline{u})$.
  It is (locally) Lipschitz-continuous, monotonically increasing, concave, and a.e. differentiable in $(0,\infty)$ with 
  \begin{align*}
  	\begin{matrix}
  	 \mathfrak{V}(\gamma)'
  	=\frac{1}{2}\sum\limits_{i=1}^m \|\partial_t \overline{u}_{\gamma,i}\|_{L^2(I)}^2 & \text{ if } \kappa=0, \text{ or}\\ \\
  	 \mathfrak{V}(\gamma)'=\frac{1}{2}\sum\limits_{i=1}^m \|\partial_t \overline{u}_{\gamma,i}\|_{L^2(I)}^2 + \frac{\kappa'(\gamma)}{2}\|\overline{u}_\gamma(0)\|^2_{\mathbb{R}^m} & \text{ if } \kappa \neq 0.
  	\end{matrix}
  \end{align*}
  \end{lemma}
  \begin{proof}
   Utilizing the fact, that $\kappa\in C^1([0,\infty))$ and $\kappa(0)=0$, the proof can be carried out along the line of \cite[Proposition 2.26]{[Piep]}.
  \end{proof}
The following theorem can be found in the context of measure valued controls in \cite[Proposition 2.27, Corollary 2.29]{[Piep]}, and for $BV-$controls in \cite[Section 6]{Hafemeyer}.
  \begin{theorem}\label{convergence_rate_reg_prob}
  The value function $\mathfrak{V}$ is continuous in $0$, and we have
 $
 0\leq J^1_\gamma(\overline{u}_\gamma)-J(\overline{u})\leq \gamma\|\mathfrak{V}'\|_{L^\infty(0,c_{loc})}
 $ for every $c_{loc}>0$, and $\gamma \in (0,c_{loc})$.
  \end{theorem}
  \begin{proof}
Let $\epsilon >0$. The space $C^\infty(I)^m$ is dense in $BV(I)^m $ with respect to the metric $d_{BV}:(\phi_1,\phi_2)\mapsto \|\phi_1-\phi_2\|_{L^2(I)^m}+|\|D_t\phi_1\|_{M(I)^m}-\|D_t\phi_2\|_{M(I)^m}|$ in $BV(I)^m $. Hence, we can find a sequence $(u_n)_n \subset C^\infty(I)^m\subset H^1(I)^m$ such that $d_{BV}(u_n,\overline{u})\xrightarrow{n\rightarrow \infty} 0$ with $\overline{u}$ as the solution of $(P)$. Due to the continuity of $\tilde{S}$, we have that $J$ is continuous with respect to the metric $d_{BV}$. The continuity of $J$ implies then, that there exists $N\in \mathbb{N}$ such that $|J(\overline{u})-J(u_n)|\leq \epsilon$ for all $n\geq N$. Thus we have for all $\gamma>0$:
\begin{align}\label{sixstar}
\begin{matrix}
\mathfrak{V}(0)=J(\overline{u})\leq J(\overline{u}_\gamma)+ \frac{\gamma}{2} \displaystyle\sum\limits_{i=1}^m \|\partial_t\overline{u}_\gamma\|_{L^2(I)^m}^2=J_\gamma(\overline{u}_\gamma)\leq J_\gamma(u_n)\\ \\=J(u_n)+\frac{\gamma}{2} \displaystyle\sum\limits_{i=1}^m \|\partial_tu_n\|_{L^2(I)^m}^2 +\frac{\kappa(\gamma)}{2} \|u_n(0)\|_{\mathbb{R}^m}^2
\leq J(\overline{u}) + \epsilon + \frac{\gamma}{2} \displaystyle\sum\limits_{i=1}^m \|\partial_tu_n\|_{L^2(I)^m}^2+\frac{\kappa(\gamma)}{2} \|u_n(0)\|_{\mathbb{R}^m}^2\\ \\
=\mathfrak{V}(0) + \epsilon + \frac{\gamma}{2} \displaystyle\sum\limits_{i=1}^m \|\partial_tu_n\|_{L^2(I)^m}^2 +\frac{\kappa(\gamma)}{2} \|u_n(0)\|_{\mathbb{R}^m}^2.
\end{matrix}
\end{align}
Because $\epsilon$ is arbitrary and $\displaystyle\sum\limits_{i=1}^m \|\partial_tu_n\|_{L^2(I)^m}^2$, and $\|u_n(0)\|_{\mathbb{R}^m}^2$ are bounded, this implies that
$
\mathfrak{V}(0)$\\
$=\liminf\limits_{\gamma \rightarrow 0}\mathfrak{V}(\gamma)=\limsup\limits_{\gamma \rightarrow 0}\mathfrak{V}(\gamma).
$
Using that $
   \mathfrak{V}(\gamma)=\int\limits_0^\gamma \mathfrak{V}'(t)dt + \mathfrak{V}(0)$ holds, we have for all $c_{loc}>0$, and $\gamma \leq c_{loc}$:
   \begin{align*}
   \begin{matrix}
   0 \leq J^1_\gamma(\overline{u}_\gamma)-J(\overline{u})=\mathfrak{V}(\gamma)-\mathfrak{V}(0) =\int\limits_0^\gamma \mathfrak{V}'(t)dt \leq \gamma\|\mathfrak{V}'\|_{L^\infty(0,c_{loc})}
   \end{matrix}
   \end{align*}
   where we used that $\mathfrak{V}$ is (locally) Lipschitz-continuous, monotonously increasing (which implies that $\mathfrak{V}'\geq 0$ a.e.), concave (which implies an a.e. decreasing derivative), and thus $\mathfrak{V}'\in L^\infty(0,c_{loc})$.

  \end{proof}

  \begin{theorem}\label{weak*conv_reg_prob}
  The unique optimal controls $\overline{u}_\gamma$ of $(P^1_\gamma)$ converge weakly* in $BV(I)^m$ to the optimal control $\overline{u}$ of $(P)$.
  \end{theorem}
  \begin{proof}
   Let $(\gamma_n)$ be an arbitrary null sequence in $\mathbb{R}^+$. In the following we show that the solutions $(\overline{u}_{\gamma_n})_{n=1}^\infty$ of the problems $(P^1_{\gamma_n})_{n=1}^\infty$ are bounded in $BV(I)^m$, with a proof which is similar to the one in \cite{[KCK]}:

  Because of the continuity of $\mathfrak{V}(\gamma)$ on $[0,\infty)$ we have that $\left( \mathfrak{V}(\gamma_n)\right)_{n=1}^\infty$ is bounded in $\mathbb{R}$. Thus, we get that $\|\partial_t\overline{u}_{\gamma_n}\|_{L^1(I)^m}=\|D_t\overline{u}_{\gamma_n}\|_{M(I)^m}$ is bounded. Next, we have to prove that $\|\overline{u}_{\gamma_n}\|_{L^1(I)^m}$ is bounded, which is required to show that $(\overline{u}_{\gamma_n})_{n=1}^\infty$ is bounded in $BV(I)^m$. Consider the decomposition $\overline{u}_{{\gamma_{n}}}=a_{\gamma_n} + \hat{u}_{\gamma_n}$ where
  \begin{align}
  \begin{matrix}
  a_{\gamma_n}=\left( a_{\gamma_n,1},...,a_{\gamma_n,m} \right), & \hat{u}_{\gamma_n}=\left( \hat{u}_{\gamma_n,1},..., \hat{u}_{\gamma_n,m} \right)\\ \\
  a_{\gamma_n}=\frac{1}{T}\int\limits_0^T \overline{u}_{\gamma_n}(t) dt \in \mathbb{R}^m, & \hat{u}_{\gamma_n} = \overline{u}_{\gamma_n} - a_{\gamma_n}.
  \end{matrix}
  \end{align}
  At first we argue that $(\overline{u}_{\gamma_n})_n$ is bounded in $BV(I)^m$.
  Note that $\|\tilde{S}(\overline{u}_{\gamma_n})-y_d\|_{L^2(\Omega_T)}^2$ is bounded, because $(v({\gamma_n}))_n$ is bounded. Thus, we get that $\tilde{S}(\overline{u}_{\gamma_n})$ is bounded in $L^2(\Omega_T)$. By (\ref{EnergyEsti}), we have that $\tilde{S}(\hat{u}_{\gamma_n})$ is bounded in $L^2(\Omega_T)$ as well, in fact
  \begin{align}\label{(E)}
  \|\tilde{S}(\hat{u}_{\gamma_n})\|_{L^2(\Omega_T)}\leq c \|\tilde{S}(\hat{u}_{\gamma_n})\|_{C(\overline{I},V)}\leq c_1(\|\hat{u}_{\gamma_n}\|_{L^1(I)^m}+c_2)\leq c_3(\|D_t\overline{u}_{\gamma_n}\|_{M(I)^m}+c_2)
  \end{align}
  where we used the BV-Poincaré inequality in the last estimate.\\
  Now define $z_n=y_n-\hat{y}_n=L(a_{\gamma_{n}}\overrightarrow{g})$ with $y_n=\tilde{S}(\overline{u}_{{\gamma_n}})$, and $\hat{y}_n=\tilde{S}(\hat{u}_{{\gamma_n}})$. The sequence $z_n$ is bounded in $L^2(\Omega_T)$.

 		To argue that $(a_{\gamma_n})_n$ is bounded we argue by contradiction, and assume that (for a subsequence, denoted by the same index) $\tilde{p}_n:=\max\limits_{1\leq j\leq m}|a_{\gamma_n,j}| \xrightarrow{n\rightarrow \infty} \infty$. Let us introduce $\xi_n=\frac{1}{\tilde{p}_n}z_n=L(\frac{1}{\tilde{p}_n}a_{\gamma_n}\overrightarrow{g})$.
 		Since
 		\begin{align}
 		\begin{matrix}
 		\| \xi_n \|_{L^2(\Omega_T)}=\| \frac{1}{\tilde{p}_n}z_n \|_{L^2(\Omega_T)} =\frac{1}{\tilde{p}_n} \| z_n \|_{L^2(\Omega_T)} \underbrace{\leq}_{z_n\text{ bdd in }L^2(\Omega_T)} \frac{1}{\tilde{p}_n} c \xrightarrow{n \rightarrow \infty} 0,
 		\end{matrix}\label{contraDi}
 		\end{align}
 		we have that $\xi_n \xrightarrow{L^2(\Omega_T)} 0$.
 		Furthermore, we have
 		\begin{align}
 		\begin{matrix}
 		\|\frac{1}{\tilde{p}_n}a_{\gamma_n}\overrightarrow{g}\|_{L^2(\Omega_T)}=\sqrt{T}\|\frac{1}{\tilde{p}_n}a_{\gamma_n}\overrightarrow{g}\|_{L^2(\Omega)}\underbrace{=}_{\text{Disj. }\supp(g_i)}\sqrt{T}\sum\limits_{i=1}^m \frac{\vert a_{\gamma_n,i}\vert}{\tilde{p}_n} \|g_i\|_{L^2(w_i)},
 		\end{matrix}\label{contraDi23}
 		\end{align}
 		which does not converge to $0$ for $n\rightarrow \infty$ since $\tilde{p}_n \rightarrow \infty$. This is a contradiction to (\ref{contraDi}) by the injectivity of the $L$ operator.
Thus we get that $(a_{\gamma_n})_n$ is a bounded sequence in $\mathbb{R}^m$ and hence $(\overline{u}_{\gamma_n})_{n=1}^\infty$ is bounded in $BV(I)^m$. Here we use that 
there exists a constant $C_T$ such that
$
 		\| u \|:= |a| + \|D_tu\|_{M(0,T)^m}\leq \max(1,T) \|u\|_{BV(0,T)^m}\leq C_T \|u\|
$
 		for all $u\in BV(0,T)^m$ and $a=\frac{1}{T}\int\limits_0^T u(t) dt \in \mathbb{R}^m$, use \cite[Theorem 3.44]{[AmFuPa]}.
Considering that bounded sequences in $BV(I)^m$ are weak* compact, we obtain by \cite[Theorem 3.23]{[AmFuPa]} that there exists a sub-sequence $(\overline{u}_{{\gamma_{n_k}}})_k$, which converges weakly* to a function $\tilde{u}\in BV(I)^m$. The weak* convergence implies that $\overline{u}_{{\gamma_{n_k}}}$ converges in $L^2(I)^m$ to $\tilde{u}$, and $D_t\overline{u}_{{\gamma_{n_k}}}$ converges in the weak* topology of $M(I)^m$ to $D_t\tilde{u}$.
  Hence, by the weak* lower semi-continuity of $\|\cdot\|_{M(I)^m}$ we get
  \begin{align}\label{weak_star_conv2}
  \begin{matrix}
  \liminf\limits_{k \rightarrow \infty} \sum\limits_{i=1}^m\alpha_i \| D_t\overline{u}_{{\gamma_{n_k}},i} \|_{M(I)}\geq \sum\limits_{i=1}^m\alpha_i \| D_t\tilde{u}_i \|_{M(I)}.
  \end{matrix}
  \end{align}
  Furthermore, the continuity of $S$ implies that
  \begin{align}\label{weak_star_conv1}
  \lim\limits_{k\rightarrow\infty}\|\tilde{S}(\overline{u}_{{\gamma_{n_k}}})-y_d\|^2_{L^2(\Omega_T)}=\|\tilde{S}(\tilde{u})-y_d\|^2_{L^2(\Omega_T)}.
  \end{align}
  Because $\|\partial_t\overline{u}_{\gamma_{n_k},i}\|_{L^2(I)}$, and $\|\overline{u}_{\gamma_{n_k},i}(0)\|_{\mathbb{R}^m}$ are bounded sequences, we have
  \begin{align}\label{weak_star_conv3}
  \begin{matrix}&
  \begin{matrix}
  \lim\limits_{k \rightarrow \infty} \frac{{\gamma_{n_k}}}{2} \sum\limits_{i=1}^m \| \partial_t\overline{u}_{{\gamma_{n_k}},i} \|_{L^2(I)}^2 = 0
     \end{matrix}
  & \text{ and for $\kappa \neq 0$} &  &
  \begin{matrix}
  \lim\limits_{k \rightarrow \infty} \frac{\kappa({\gamma_{n_k}})}{2} \| \overline{u}_{{\gamma_{n_k}}}(0) \|_{\mathbb{R}^m}^2 = 0.
  \end{matrix}&
  \end{matrix}
  \end{align}
  Estimates (\ref{weak_star_conv2}) - (\ref{weak_star_conv3}) and Theorem \ref{convergence_rate_reg_prob} imply that
  \begin{align*}
  \begin{matrix}
  \mathfrak{V}(0)=\lim\limits_{\gamma_{n_k}\rightarrow 0^+}\begin{Bmatrix}
   \frac{1}{2} \|\tilde{S}(\overline{u}_{{\gamma_{n_k}}})-y_d\|^2_{L^2(\Omega_T)} +  \sum\limits_{i=1}^m\alpha_i \| \partial_t\overline{u}_{{\gamma_{n_k}},i} \|_{M(I)}\\ \\
   +  \frac{{\gamma_{n_k}}}{2} \sum\limits_{i=1}^m \| \partial_t\overline{u}_{{\gamma_{n_k}},i} \|_{L^2(I)}^2
   +\frac{\kappa({\gamma_{n_k}})}{2} \| \overline{u}_{{\gamma_{n_k}}}(0) \|_{\mathbb{R}^m}^2
  \end{Bmatrix}\\ \\
  \geq\frac{1}{2} \|\tilde{S}(\tilde{u})-y_d\|^2_{L^2(\Omega_T)} +\sum\limits_{i=1}^m\alpha_i \| D_t\tilde{u}_{i} \|_{M(I)}=J(\tilde{u}).
  \end{matrix}
  \end{align*}
  By uniqueness of the optimal control of $(P)$ we get that $\tilde{u}$ is equal to the optimal control $\overline{u}$ of $(P)$. Thus, the unique solutions $\overline{u}_{{\gamma_{n_k}}}$ of $(P^1_{{\gamma_{n_k}}})$ converge $BV(I)^m$-weak* to the optimal control $\overline{u}$ of $(P)$.
  \end{proof}
  \begin{corollary}\label{strict_conv_BV}
  The unique optimal controls $\overline{u}_\gamma$ of $(P^1_\gamma)$ converge strictly in $BV(I)^m$ to the optimal control $\overline{u}$ of $(P)$.
  \end{corollary}
  \begin{proof}
  Due to the weak* convergence by Theorem \ref{weak*conv_reg_prob} we get that $\overline{u}_\gamma$ converges in $L^1(I)^m$ to the optimal control $\overline{u}$. Using that $\tilde{S}(\overline{u}_\gamma)\rightarrow \tilde{S}(\overline{u})$ in $L^2(\Omega_T)$, Theorem \ref{convergence_rate_reg_prob} implies that the total variations of $\overline{u}_\gamma$ converge to the total variation of $\overline{u}$.
  \end{proof}

 \subsection{Equivalent Regularized Optimal Control Problem to $(P^1_\gamma)$}\label{Chapter_of_Probelm_P1_gamma}
 In this section we introduce an equivalent problem $(\tilde{P}_\gamma)$ to $(P^1_\gamma)$. The latter will be solved by a semi-smooth Newton method. In the remaining part of the paper we restrict the operator $B$ defined in (\ref{B_red_to_L_2}) to $L^2(I)^m\times \mathbb{R}^m$.
 Its adjoint has the form
 \begin{align*}
 \begin{matrix}
 B^*: L^2(\Omega_T) & \rightarrow & L^2(I)^m\times \mathbb{R}^m,  & &
 \varphi & \mapsto & \left( \begin{matrix} \int\limits_{\Omega}\int\limits_\cdot^T \overrightarrow{g} \varphi \\ \\ \int\limits_{\Omega_T} \overrightarrow{g} \varphi \end{matrix} \right)(s)
 \end{matrix}
 \end{align*}
Analogously we henceforth restrict $S$ to $L^2(I)^m\times\mathbb{R}^m$. The isomorphism in Lemma \ref{isoBV} translates $(P^1_\gamma)$ into the following equivalent form:
 \begin{align*}
 (\tilde{P}_\gamma)\left\lbrace \begin{matrix}
 & \min\limits_{(v,c)\in L^2(I)^m\times\mathbb{R}^m} \begin{Bmatrix}
 \frac{1}{2}\| S(v,c)-y_d \|_{L^2(\Omega_T)}^2 + \sum\limits_{j=1}^m \alpha_j \|v_j\|_{L^1(I)}\\ \\ + \frac{\gamma}{2} \sum\limits_{j=1}^m \|v_j\|_{L^2(I)}^2+\frac{\kappa(\gamma)}{2} \|c\|_{\mathbb{R}^m}^2
 \end{Bmatrix}=:\tilde{J}_\gamma(v,c)
 \end{matrix}\right.
 \end{align*}
 \subsection{Regularization - First-Order Optimality Condition}
 In this section we present the first-order optimality conditions for $(\tilde{P}_\gamma)$. We will use a prox-operator approach to represent implicitly the distributional derivative of the BV-optimal control of $(P)$ with respect to the adjoint. This allows to replace the sub-differential in the first-order optimality conditions of $(\tilde{P}_\gamma)$. Finally we compare the sparsity results of $(\tilde{P}_\gamma)$ and $(\tilde{P})$, and show the convergence of the adjoints of $(\tilde{P}_\gamma)$ to the adjoint of $(\tilde{P})$ for $\gamma \rightarrow \infty$.

 \begin{lemma}\label{maxminFormula1}
 Let $(\overrightarrow{v},\overrightarrow{c})\in L^2(I)^m \times \mathbb{R}^m$ be the optimal control of $(\tilde{P}_\gamma)$. We have the following necessary and sufficient optimality conditions for $(\tilde{P}_\gamma)$:
   \begin{align*}
   \begin{matrix}
   (1) & \overrightarrow{v}(s)=\left(\begin{matrix}\max
   \left(0,-\frac{1}{\gamma}\int\limits_{\Omega}\int\limits_{s}^T
   L^*(S(\overrightarrow{v},\overrightarrow{c})-y_d)g_i dtdx - \frac{\alpha_i}{\gamma}\right)+\\
   +\min\left(0,-\frac{1}{\gamma}\int\limits_{\Omega}\int\limits_{s}^T
   L^*(S(\overrightarrow{v},\overrightarrow{c})-y_d)g_i dtdx + \frac{\alpha_i}{\gamma}\right)\end{matrix} \right)_{i=1}^m\in L^2(I)^m \\ \\
   (2) & -\int_{\Omega_T}L^*(S(\overrightarrow{v},\overrightarrow{c})-y_d)\overrightarrow{g} dtdx-\kappa(\gamma) \overrightarrow{c}=0_{\mathbb{R}^m} \in \mathbb{R}^m
   \end{matrix}
   \end{align*}
 \end{lemma}
\begin{proof}
Since this proof is standard, we have deferred it to the appendix.
\end{proof}
In the appendix it is also shown that  \begin{align}\label{proxerop}
  \begin{matrix}
  \prox^\gamma_{\sum\limits_i\alpha_i \|\cdot\|_{L^1(I)}}\left(-\frac{1}{\gamma}\int\limits_{\Omega}\int\limits_s^T L^*(S(\overrightarrow{v},\overrightarrow{c})-y_d) \overrightarrow{g} dt dx\right)
  \end{matrix}
  \end{align}
is equal to the right hand side of equation (1) in Lemma \ref{maxminFormula1}.

Due to the regularity of the adjoint wave equation, we have that the optimal control $\overrightarrow{v}$ is at least Lipschitz continuous.
%
 \begin{proposition}\label{key123}
 Let $(\overrightarrow{v},\overrightarrow{c})\in L^2(I)^m \times \mathbb{R}^m$ be the optimal control of $(\tilde{P}_\gamma)$. Then we have for a.a. $s\in I$ and $i=1,...,m$:
  \begin{align}\label{implicitcontrol}
  \begin{matrix}
  \overrightarrow{v}_i(s)=\left\lbrace\begin{matrix}
  0 & , & |\psi_{\gamma,i}(s)|<\alpha_i\\ \\
  -\frac{1}{\gamma}\psi_{\gamma,i}(s) + \frac{\alpha_i}{\gamma} & , & \psi_{\gamma,i}(s)\geq \alpha_i\\ \\
  -\frac{1}{\gamma}\psi_{\gamma,i}(s) - \frac{\alpha_i}{\gamma} & , & \psi_{\gamma,i}(s)\leq -\alpha_i
  \end{matrix}\right.
  \end{matrix}
  \end{align}
  with $\psi_\gamma(s)=\int\limits_{\Omega}\int\limits_s^T L^*(S(\overrightarrow{v},\overrightarrow{c})-y_d)\overrightarrow{g} dtdx$, and $\psi_{\gamma,i}(s)=\int\limits_{\Omega}\int\limits_s^T L^*(S(\overrightarrow{v},\overrightarrow{c})-y_d)g_i dtdx$.
 \end{proposition}
 One can compare the sparsity structure of the optimal controls associated to $(\tilde{P}_\gamma)$ to the sparsity for the optimal control of $(\tilde{P})$. The optimal measures $\overrightarrow{v}_i$ in $(\tilde{P})$, see Lemma \ref{support_measure_cor}, are not supported, where $\vert p_{1,i}(t) \vert <\alpha_i$, while the optimal measures for $(\tilde{P}_\gamma)$ are not supported, where $\vert \psi_{\gamma,i}(t) \vert < \alpha_i $ holds.

We next address the convergence of the adjoints $\psi_\gamma$ of $(\tilde{P}_\gamma)$ to the adjoint $p_1$ of $(\tilde{P})$, which is defined in Theorem \ref{FirstOrder}.
 \begin{proposition}\label{DoesNotWorkWithKappaBigger0}
 For $0<\gamma \rightarrow 0$ we find
 $\psi_\gamma \xrightarrow{H^2(I)^m}p_1$. Furthermore, we have for $\kappa=0$,
 \begin{align}\label{E1Korr}
 	\begin{matrix}
 	\left\langle -\frac{\psi_{\gamma,i}}{\alpha_i},\overrightarrow{v}_{\gamma,i}\right\rangle_{C_0(I),M(I)} \xrightarrow{\gamma \rightarrow 0}
 	\left\langle -\frac{p_{1,i}}{\alpha_i},\overrightarrow{v}_{i}\right\rangle_{C_0(I),M(I)} = \|\overrightarrow{v}_i\|_{M(I)},
 		\end{matrix}
 \end{align}
\begin{align}\label{E2Korr}
	\begin{matrix}
	\text{ and for }\kappa> 0\text{, }\int\limits_I -\frac{\psi_{\gamma,i}}{\alpha_i}d\overrightarrow{v}_{\gamma,i}(s) \xrightarrow{\gamma \rightarrow 0}
	\left\langle -\frac{p_{1,i}}{\alpha_i},\overrightarrow{v}_{i}\right\rangle_{C_0(I),M(I)} = \|\overrightarrow{v}_i\|_{M(I)},
	\end{matrix}
\end{align}
    for $i=1,\cdots,m,$ where $(\overrightarrow{v},\overrightarrow{c})\in M(I)^m\times\mathbb{R}^m$ is the solution to $(\tilde{P})$, and $(\overrightarrow{v}_\gamma,\overrightarrow{c}_\gamma)\in L^2(I)^m\times\mathbb{R}^m$ is the solution to $(\tilde{P}_\gamma)$.
 \end{proposition}
 Due to Theorem \ref{weak*conv_reg_prob} we know that $\overrightarrow{v}_\gamma$, the derivative of the optimal control $\overline{u}_\gamma$ of $(P^1_\gamma)$, converges weakly* to $\overrightarrow{v}$ in $M(I)^m$, the distributional derivative of the optimal control $\overline{u}$ of $(P)$. Furthermore, recall that $\left\langle -\frac{p_{1,i}}{\alpha_i},\overrightarrow{v}_{i}\right\rangle_{C_0(I),M(I)} = \|\overrightarrow{v}_i\|_{M(I)}$ holds due to Lemma \ref{support_measure_cor}.

\begin{proof}
  We first show that $\|\psi_{\gamma,i }- p_{1,i}\|_{H^2(I)}\xrightarrow{\gamma \rightarrow 0} 0$, for $i=1,\cdots,m$, holds.
  By regularity results of the wave equation we have that $p_{1,i}$ and $\psi_{\gamma,i}$ are elements of $H^2(I)$.
Furthermore, using Theorem \ref{weak*conv_reg_prob} in the last inequality of the following computation we find
  \begin{align*}
  \begin{matrix}
  \|\partial_t \psi_{\gamma,i } - \partial_t p_{1,i}\|_{L^2(I)}^2=\left\|
  \int\limits_{\Omega}L^*(L((\overline{u}_\gamma-\overline{u})\cdot\overrightarrow{g}))g_idx
  \right\|_{L^2(I)}^2\leq \int\limits_0^T \left(
  \int\limits_\Omega \vert L^*(L((\overline{u}_\gamma-\overline{u})\cdot\overrightarrow{g}))\vert \cdot \vert g_i\vert dx
  \right)^2dt\\ \\
  \leq
  \int\limits_0^T \left(
  \| L^*(L((\overline{u}_\gamma-\overline{u})\cdot\overrightarrow{g}))\|_{H} \cdot \| g_i\|_{H}
  \right)^2dt= c\|L^*(L((\overline{u}_\gamma-\overline{u})\cdot\overrightarrow{g}))\|_{L^2(I;H)}^2\\ \\
  \leq c\|L^*(L((\overline{u}_\gamma-\overline{u})\cdot\overrightarrow{g}))\|_{C(\overline{I};H)}^2
  \underbrace{\leq}_{(\ref{EnergyEsti})}
  c\|(\overline{u}_\gamma-\overrightarrow{u})\cdot\overrightarrow{g}\|_{L^1(I;H)}^2\leq c\|\overline{u}_\gamma-\overline{u}\|_{L^1(I)}^2\underbrace{\xrightarrow{\gamma \rightarrow 0}}_{\begin{matrix}
  	\overline{u}_\gamma \xrightarrow{BV,w^*}\overline{u}
  	\end{matrix}}0.
  \end{matrix}
  \end{align*}
Since $\psi_{\gamma,i}-p_{1,i}\in H^1(I)$ with $(\psi_{\gamma,i}-p_{1,i})(T)=0$ this implies that $\|\psi_{\gamma,i }- p_{1,i}\|_{H^1(I)}\xrightarrow{\gamma \rightarrow 0}0$.
 Let us show that $\|\partial_{tt}\psi_{\gamma,i }- \partial_{tt}p_{1,i}\|_{L^2(I)}\xrightarrow{\gamma \rightarrow 0}0$ holds as well.
  For this purpose, utilizing the dominated convergence theorem and Theorem \ref{weak*conv_reg_prob} we obtain
  \begin{align*}
  \begin{matrix}
  \|\partial_{tt} \psi_{\gamma,i } - \partial_{tt} p_{1,i}\|_{L^2(I)}^2=\left\|\partial_t
  \int\limits_{\Omega}L^*(L((\overline{u}_\gamma-\overline{u})\cdot\overrightarrow{g}))g_idx
  \right\|_{L^2(I)}^2\\ \\
  =
  \int\limits_0^T \left(
  \int\limits_\Omega\partial_t L^*(L((\overline{u}_\gamma-\overline{u})\cdot\overrightarrow{g}))g_idx
  \right)^2dt\leq \int\limits_0^T \left(
  \int\limits_\Omega \vert \partial_tL^*(L((\overline{u}_\gamma-\overline{u})\cdot\overrightarrow{g}))\vert \cdot \vert g_i\vert dx
  \right)^2dt\\ \\
  \leq
  \int\limits_0^T \left(
  \| \partial_tL^*(L((\overline{u}_\gamma-\overline{u})\cdot\overrightarrow{g}))\|_{H} \cdot \| g_i\|_{H}
  \right)^2dt= c\|\partial_tL^*(L((\overline{u}_\gamma-\overline{u})\cdot\overrightarrow{g}))\|_{L^2(I;H)}^2
      \end{matrix}
  \end{align*}
  \begin{align*}
  \begin{matrix}
  \leq c\|\partial_tL^*(L((\overline{u}_\gamma-\overline{u})\cdot\overrightarrow{g}))\|_{C(\overline{I};H)}^2
  \underbrace{\leq}_{(\ref{EnergyEsti})}
  c\|(\overline{u}_\gamma-\overline{u})\cdot\overrightarrow{g}\|_{L^1(I;H)}^2\leq c\|\overline{u}_\gamma-\overline{u}\|_{L^1(I)^m}^2\xrightarrow{\gamma \rightarrow 0}0.
  \end{matrix}
  \end{align*}
%
%
Consider now the case $\kappa=0$:
To verify (\ref{E1Korr}) let us note that $\psi_{\gamma,i}\in H^1_0(I)$ since $\kappa=0$.
 Because $H^2(I)$ continuously embeds into $C^0(\overline{I})$, and $\psi_{\gamma,i},p_{1,i}\in C_0(I)$ we have $\|\psi_{\gamma,i} -p_{1,i}\|_{C_0(I)}\xrightarrow{\gamma \rightarrow 0}0$, for $i=1,\cdots,m$. 
 Utilizing the weak* convergence of $\overrightarrow{v}_{\gamma,i}\xrightarrow{w^*,M(I)}\overrightarrow{v}_i$ and the strong convergence $\psi_{\gamma,i} \xrightarrow{C_0(I)}p_{1,i}$, for $i=1,\cdots,m$, we achieve the desired result.

 We turn to verify (\ref{E2Korr}). Since $\kappa\neq 0$ we do not have that $\psi_{\gamma,i}\in H^1_0(I)$.
 Consider the following function
$
  \varphi(t)=cos(\frac{\pi}{T}t) \mathbf{1}_{\left[0,\frac{T}{2}\right)}(t)
$,
which satisfies $\varphi(0)=1$, $\varphi(t)=0$, for $t\geq \frac{T}{2}$, $\varphi \in C(\overline{I})$ and thus
 \begin{align}\label{E3Korr}
  \begin{matrix}
  \int\limits_I\psi_{\gamma,i}d\overrightarrow{v}_{\gamma,i}= \left\langle
  \psi_{\gamma,i}+\kappa(\gamma)c_{\gamma,i}\varphi,\overrightarrow{v}_{\gamma,i}\right\rangle_{C_0,M(I)}-\int\limits_I\kappa(\gamma)\overrightarrow{c}_{\gamma,i}\varphi d\overrightarrow{v}_{\gamma,i}.
  \end{matrix}
 \end{align}
 Furthermore, we have by the convergence of $\psi_{\gamma,i} \xrightarrow[\gamma\rightarrow 0]{H^2(I)}p_{1,i}$, for $i=1,\cdots,m$, and by the embedding $H^2(I)\hookrightarrow C(\overline{I})$ that $\psi_{\gamma,i} \xrightarrow[\gamma\rightarrow 0]{C(\overline{I})}p_{1,i}$. Due to weak* convergence $\overline{u}_\gamma\xrightarrow[BV(I)^m]{w^*}\overline{u}$, we have that $(\overline{u}_\gamma)_\gamma$ is bounded in $BV(I)^m$ for $\gamma\rightarrow 0$. By the isomorphism in Lemma \ref{isoBV} we get that $(D_t\overline{u}_\gamma,\overline{u}_\gamma(0))=(\overrightarrow{v}_\gamma,\overrightarrow{c}_\gamma)$  is bounded in $M(I)^m \times \mathbb{R}^m$ as well.
 Given the boundedness of $\vert \overrightarrow{c}_{\gamma,i}\vert$ it holds that $\|\kappa(\gamma)\overrightarrow{c}_{\gamma,i}\varphi\|_{\infty}\leq \kappa(\gamma)\vert \overrightarrow{c}_{\gamma,i}\vert \|\varphi\|_{\infty}\xrightarrow{\gamma
 \rightarrow 0}0$. Summing up, we have
$
  \psi_{\gamma,i} \xrightarrow[\gamma\rightarrow 0]{C(\overline{I})}p_{1,i}$ and $\kappa(\gamma)\overrightarrow{c}_{\gamma,i}\varphi\xrightarrow[\gamma\rightarrow 0]{C(\overline{I})}0,
  $
 which implies $\psi_{\gamma,i} +\kappa(\gamma)\overrightarrow{c}_{\gamma,i}\varphi \xrightarrow[\gamma\rightarrow 0]{C_0(\overline{I})}p_{1,i}$.
 Together with weak* convergence in $M(I)$ of $\overrightarrow{v}_{\gamma,i}$ to $\overrightarrow{v}_i$, we get
 \begin{align}\label{E4Korr}
  \begin{matrix}
  \left\langle
  \psi_{\gamma,i}+\kappa(\gamma)\overrightarrow{c}_{\gamma,i}\varphi,\overrightarrow{v}_{\gamma,i}\right\rangle_{C_0,M(I)} \xrightarrow{\gamma \rightarrow 0} \left\langle p_{1,i},\overrightarrow{v}_{i}\right\rangle_{C_0(I),M(I)}.
  \end{matrix}
 \end{align}
 Due to the boundedness of $(\overrightarrow{v}_{\gamma})_\gamma$ in $M(I)^m$ and $(\overrightarrow{c}_\gamma)_\gamma$ in $\mathbb{R}^m$, for $\gamma\rightarrow 0$, we get
 \begin{align}\label{key23wa}
  \begin{matrix}
  \left\vert\int\limits_I \kappa(\gamma)\overrightarrow{c}_{\gamma,i}\varphi d\overrightarrow{v}_{\gamma,i}\right\vert \leq \int\limits_I \vert \kappa(\gamma)\overrightarrow{c}_{\gamma,i}\varphi\vert  d\overrightarrow{v}_{\gamma,i} \underbrace{\leq}_{\vert\varphi\vert\leq 1} \kappa(\gamma) \vert \overrightarrow{c}_{\gamma,i} \vert \|\overrightarrow{v}_{\gamma,i}\|_{M(I)} \xrightarrow{\gamma \rightarrow 0} 0.
  \end{matrix}
 \end{align}
	Finally, we consider (\ref{E4Korr}), and (\ref{key23wa}) in (\ref{E3Korr}) and get (\ref{E2Korr}).
%
 \end{proof}

	\subsection{Regularization - Semi-smooth Newton Method}
	\label{SuperLinSection}
	In this section, we discuss the semi-smooth Newton method which is used to construct a sequence in $L^2(\Omega)^m\times\mathbb{R}^m$ that solves the first-order condition (1), (2) in Lemma \ref{maxminFormula1} in the limit. Later in section \ref{numericsection} a BV-path following algorithm is presented which uses these method, see Algorithm \ref{Algo1}.

	At first, let us introduce 
		\begin{align*}
		\begin{matrix}
			F_{\gamma}:L^2(I)^m\times \mathbb{R}^m \rightarrow L^2(I)^m\times \mathbb{R}^m, & &
			F_{\gamma}(v,c):=\left(\begin{matrix}
				v-\prox^\gamma_{\sum\limits_i\alpha_i \|\cdot\|_{L^1(I)}}(-\frac{1}{\gamma} \tilde{\pi}(v,c))\\
				\frac{\kappa(\gamma)}{\gamma}c+\frac{1}{\gamma}\tilde{\pi}(v,c)(0)
			\end{matrix}\right)
		\end{matrix}
	\end{align*}
	where
	\begin{align}\label{ProxDef}
	\begin{matrix}
	\prox^\gamma_{\sum\limits_i\alpha_i \|\cdot\|_{L^1(I)}}(-\frac{1}{\gamma} \tilde{\pi}(v,c))=\left(\begin{matrix}\max
	\left(0,-\frac{1}{\gamma} \tilde{\pi}_{i}(v,c) - \frac{\alpha_i}{\gamma}\right)+\\
	+\min\left(0,-\frac{1}{\gamma}\tilde{\pi}_{i}(v,c) + \frac{\alpha_i}{\gamma}\right)\end{matrix} \right)_{i=1}^m\in L^2(I)^m,
	\\ \\
	\tilde{\pi}(v,c)(s)=\int\limits_{\Omega}\int\limits_s^T L^*(S(v,c)-y_d)\overrightarrow{g} dtdx\in L^2(I)^m,
	\end{matrix}
	\end{align}
	and observe that $F_\gamma(\overrightarrow{v},\overrightarrow{c})=0$ is equivalent to (1), (2) in Lemma \ref{maxminFormula1}.

	Consider the following definition of \cite[p. 120 et seq.]{[HPUU]}:
	\begin{definition} Let $G:X\rightarrow Y$ be a continuous operator, between Banach spaces $X$ and $Y$. Further, let us consider a set-valued mapping $\partial G: X \rightrightarrows L(X,Y)$ with non-empty images. We call $\partial G$ a generalized differential. We define the operator $G$ to be $\partial G$-semi-smooth or simply semi-smooth in $x^*$, if
			\begin{align*}
			\begin{matrix}
			\sup\limits_{M\in \partial G(x^*+d)} \| G(x^*+d) -G(x^*)-Md \|_Y=o(\| d \|_X) \text{ for }\|d\|_X\rightarrow 0
			\end{matrix}
			\end{align*}
	\end{definition}
	We recall the following theorem from \cite[Theorem 1.1]{[hik]}:
	\begin{theorem}\label{superlin}
	Suppose that $x^*$ is a solution of the equation $G(x^*)=0$ and that $G$ is $\partial G$-semi-smooth in a neighborhood $U$ in $X$ containing $x^*$. If the set $\partial G(x)$ contains only non-singular mapping and if $\lbrace \|M^{-1}\|\text{ }\vert\text{ } M\in \partial G(x) \rbrace $ is bounded for all $x\in U$, then the Newton iteration
	\begin{align}\label{NewtonAlgo}
	\begin{matrix}
	x^{k+1}=x^k-M^{-1}G(x^k), \text{ for any $M\in \partial G(x^k)$}
	\end{matrix}
	\end{align}
	converges super linearly to $x^*$, provided that $\|x^0 - x^*\|$ is sufficiently small.
	\end{theorem}
	\begin{remark}
	Let us note that \cite[Theorem 1.1]{[hik]} is actually more general. The authors are using the slant differentiability which is a weaker concept than the semi-smoothness, see \cite[p. 868]{[hik]}.
	\end{remark}
	In the following, we prove that all conditions needed for Theorem \ref{superlin} hold for $G=F_\gamma$ and $\kappa\neq 0$ with $x^*=(\overrightarrow{v},\overrightarrow{c})$ and $X=Y=L^2(I)^m\times \mathbb{R}^m$. If the initial value $x^0$ is sufficiently close to $x^*$ this guarantees that the sequence $(x^k)_{k\in \mathbb{N}}$ in (\ref{NewtonAlgo}) converges super linearly in $L^2(I)^m\times \mathbb{R}^m$ to $x^*$ with respect to $F_\gamma$.

\begin{definition}
	Define the following operators for $(\overrightarrow{h},\overrightarrow{k})\in L^2(I)^m\times
	\mathbb{R}^m$:
	\begin{align*}
	\begin{matrix}
	(B^*L^*LB)_1(\cdot,\overrightarrow{k}): & L^2(I)^m & \rightarrow &L^2(I)^m, & \text{ by }
	& \overrightarrow{h} & \mapsto & \int\limits_{\cdot}^T\int\limits_{\Omega}L^*(L(B(\overrightarrow{h},\overrightarrow{k})))\overrightarrow{g}dtdx& \\ \\
	(B^*L^*LB)_2(\overrightarrow{h},\cdot): & \mathbb{R}^m & \rightarrow &\mathbb{R}^m, & \text{ by }
	& \overrightarrow{k} & \mapsto & \int\limits_{\Omega_T}L^*(L(B(\overrightarrow{h},\overrightarrow{k})))\overrightarrow{g}dtdx.
	\end{matrix}
	\end{align*}
	Furthermore, we write for $i=1,\cdots,m$:
	\begin{align*}
	\begin{matrix}
	(B^*L^*LB)_{1,i}(\overrightarrow{h},\overrightarrow{k})=\int\limits_{\Omega}\int\limits_s^T L^*(L(B(\overrightarrow{h},\overrightarrow{k})))g_i dtdx \in L^2(I),\\ \\
	(B^*L^*LB)_{2,i}(\overrightarrow{h},\overrightarrow{k})=\int\limits_{\Omega_T} L^*(L(B(\overrightarrow{h},\overrightarrow{k})))g_i dtdx \in \mathbb{R}.
	\end{matrix}
	\end{align*}
\end{definition}
		
For any function $\Upsilon:X\rightarrow Y$, with $X$ and $Y$ Banach spaces, we denote by $D\Upsilon(x)(z)$ the directional derivative of $\Upsilon$ in $x$ in direction $z$.

%
%
%
%
%

	Let us recall that the point-wise maximum and minimum operation from $L^p$ to $L^2$ are semi-smooth if $p>2$ (Norm gap), and a Newton derivative in $f\in L^p(I)$ in direction $h\in L^p(I)$ is given by
		\begin{align*}
		\begin{matrix}
		\left\lbrace \begin{matrix}
		\mathds{1}_{(0,\infty)}(f)(h)\vert_{h\in L^p(I)} & \text{ for }\max,\\ \\
		\mathds{1}_{(-\infty,0)}(f)(h)\vert_{h\in L^p(I)} & \text{ for }\min.
		\end{matrix}\right.
		\end{matrix}
		\end{align*}
		Hence, we get for $	\tilde{\max}/\tilde{\min}:L^p(I)^m \rightarrow L^2(I)^m$, $\tilde{\max}(\overrightarrow{f}):=(\max(f_i))_{i=1,...,m}$ respectively $\tilde{\min}(\overrightarrow{f})$ $:=(\min(f_i))_{i=1,...,m}$, the following Newton derivative in $\overrightarrow{f}\in L^p(I)^m$ in direction $\overrightarrow{h}\in L^p(I)^m$:
		\begin{align*}
		\begin{matrix}
		D\tilde{\max}/\tilde{\min}(\overrightarrow{f})(\overrightarrow{h})=\left(\begin{matrix}
		\mathds{1}_{\mathcal{J}}(f_1) & 0 & ... & ... & 0\\
		0 & \mathds{1}_{\mathcal{J}}(f_2) & 0  & ... & 0\\
		0 & 0 & . & ... & .\\
		. & . & . & ... & .\\
		. & . & . & \mathds{1}_{\mathcal{J}}(f_{m-1}) & 0 \\
		0 & 0 & ... & 0 & \mathds{1}_{\mathcal{J}}(f_{m})
		\end{matrix} \right) \left( \begin{matrix}
		h_1\\
		.\\
		.\\
		.\\
		h_{m-1}\\
		h_m
		\end{matrix}\right)\\ \\
		:=\mathds{1}_{\mathcal{J}}(\overrightarrow{f})(\overrightarrow{h}):=diag((\mathds{1}_{\mathcal{J}}f_i)_i)\overrightarrow{h}
		\end{matrix}
		\end{align*}
		for $\mathcal{J}=(0,\infty)$ in case of $max$ and $(-\infty,0)$ in case of $\min$. The matrix $\mathds{1}_{\mathcal{J}}(\overrightarrow{f})$ has only values equal to 1 or 0 on its diagonal.
		\begin{lemma}\label{key1245}\label{Case2Df}
			The first derivative of $F_{\gamma}$ in $(v,c)\in L^2(I)^m\times \mathbb{R}^m$ has the following form:
			\begin{align*}
			\begin{matrix}
			DF_{\gamma}(v,c)=\begin{pmatrix}
			id_{L^2(I)^m} & 0\\
			0 & \frac{\kappa(\gamma)}{\gamma}id_{\mathbb{R}^m}
			\end{pmatrix}+\frac{1}{\gamma}\begin{pmatrix}
			\overline{\diag} & 0\\
			0 & id_{\mathbb{R}^m}
			\end{pmatrix}B^*L^*LB\\ \\
			=\begin{pmatrix}
			id_{L^2(I)^m} +\frac{1}{\gamma}\overline{\diag}\cdot(B^*L^*LB)_1(\cdot,0) & & & \frac{1}{\gamma}\overline{\diag}\cdot(B^*L^*LB)_1(0,\cdot) \\ \\
			\frac{1}{\gamma}(B^*L^*LB)_2(\cdot,0) & & & \frac{\kappa(\gamma)}{\gamma} id_{\mathbb{R}^m}+\frac{1}{\gamma}(B^*L^*LB)_2(0,\cdot)
			\end{pmatrix}\\ \\
						 \text{ with }
				\end{matrix}
			\end{align*}
			\begin{align}
			\begin{matrix}
		 \overline{\diag}:=\left(\begin{matrix}\mathds{1}_{(0,\infty)}\left( -\frac{1}{\gamma}  \tilde{\pi}(v,c)-\frac{\overrightarrow{\alpha}}{\gamma}\right)
			+\mathds{1}_{(-\infty,0)}\left( -\frac{1}{\gamma}  \tilde{\pi}(v,c)+\frac{\overrightarrow{\alpha}}{\gamma}\right)\end{matrix}\right).
			\end{matrix}
			\end{align}
			In particular, we have for $(\overrightarrow{h},\overrightarrow{k})\in L^2(I)^m \times \mathbb{R}^m$:
			\begin{align}\label{newProof}
			\begin{matrix}
			DF_{\gamma}(v,c)\left( \begin{matrix}
			\overrightarrow{h}\\ \overrightarrow{k}\end{matrix}\right)(s) =  \left(\begin{matrix}
			\overrightarrow{h}(s)+\overline{\diag}\cdot
			\left( \frac{1}{\gamma}  \int\limits_{\Omega}\int\limits_s^T L^*(L(B(\overrightarrow{h},\overrightarrow{k})))\overrightarrow{g}dtdx \right)\\ \\
			\frac{\kappa(\gamma)}{\gamma}\overrightarrow{k}+\frac{1}{\gamma}\int\limits_{\Omega_T}L^*(L(B(\overrightarrow{h},\overrightarrow{k})))\overrightarrow{g}dtdx
			\end{matrix} \right).
			\end{matrix}
			\end{align}
			Furthermore, the function $F_\gamma$ with $DF_\gamma(v,c)$ as generalized derivative is semi-smooth for all $ (v,c)\in L^2(I)^m\times \mathbb{R}^m$.
 		\end{lemma}
 	\begin{proof}
 		Lemma \ref{key1245} is a consequence of the semi-smoothness of $\max/\min$ from $L^p$ to $L^q$ with $p>q\geq 1$.
 	\end{proof}
	 	Let us write $\overline{diag}$ for $\overline{diag}(s):=diag(X_1(s),\cdots,X_m(s))$ with $X_i:I\rightarrow \lbrace 0,1\rbrace$. The maps $X_i$ are $\mathcal{B}(I)-2^{\lbrace 0,1\rbrace}$-measurable. Hence, we can define the measurable sets $I_{0,i}:=\lbrace s\in I \vert  X_i(s)=0\rbrace,I_{1,i}:=\lbrace s\in I \vert X_i(s)=1 \rbrace$.
	\begin{lemma}\label{key28}
	The linear continuous operator
		\begin{align*}
		\begin{matrix}
		B^*L^*LB:L^2(I)^m \times \mathbb{R}^m \rightarrow L^2(I)^m \times \mathbb{R}^m
		\end{matrix}
		\end{align*}
	is a self-adjoint non-negative and injective operator with spectrum inside $[0,\|B^*L^*LB\|]$.

	Furthermore, $G:=(\langle Lg_i,Lg_j\rangle_{L^2(\Omega_T)})_{i,j=1}^m$ is invertible and we have for all $h\in L^2(I)^m$ that the continuous affine linear  operator $(B^*L^*LB)_2(h,\cdot)$ is bijective with
		\begin{align}\label{key23}
		\begin{matrix}
		(B^*L^*LB)_2(h,k)=\phi \Leftrightarrow k=G^{-1}(\phi-(B^*L^*LB)_2(h,0)).
		\end{matrix}
		\end{align}
	\end{lemma}
	\begin{proof}
	The non-negativity and injectivity can be seen by the strict inequality,
	\begin{align*}
	\begin{matrix}
		\left\langle
		B^*L^*LB\phi,\phi
		\right\rangle_{L^2(I)^m\times \mathbb{R}^m}=\left\langle 	LB\phi,LB\phi
			\right\rangle_{L^2(\Omega_T)}> 0 \text{ for }\phi\neq 0.
	\end{matrix}
	\end{align*}
	The strictness is a consequence of the uniqueness of solutions of the wave equation defined by $L$.

	The claim on the spectrum follows from selfadjointness and the fact that the spectral radius is $\|B^*L^*LB\|$, see \cite[Theorem VI.6]{ReedSim}.
%
%
%
%

	Let us now show that $G$ is invertible.
	Given the linear independence of $(g_i)_{i=1}^m$ in $L^2(\Omega_T)$ we get that $(L(g_i))_{i=1}^m$ is linearly independent in $L^2(\Omega_T)$ by the uniqueness of solutions of the wave equation.
	Further, introduce
$
			\langle \lambda,\mu\rangle_L=
		\langle L(\lambda\cdot \overrightarrow{g}),
			L(\mu\cdot \overrightarrow{g})\rangle_{L^2(\Omega_T)}
			=\langle \sum\limits_{i=1}^m\lambda_iL(g_i),\sum\limits_{j=1}^m\mu_jL(g_j)\rangle_{L^2(\Omega_T)}
$. This is an inner product in $\mathbb{R}^m$.
Hence, the Gram-Schmidt Matrix $G=(\langle e_i,e_j\rangle_L)_{i,j=1}^m\in \mathbb{R}^{m\times m}$ is invertible.

To verify (\ref{key23})	let us derive that
\begin{align*}
	\begin{matrix}
	(B^*L^*LB)_2(0,k)=\left(\langle L^*L(k\cdot \overrightarrow{g}),g_i\rangle_{L^2(\Omega_T)}\right)_{i=1}^m=\left(
	\left\langle
	\sum\limits_{j=1}^m k_j L(g_j),L(g_i)
	\right\rangle_{L^2(\Omega_T)}
	\right)_{i=1}^m=G(c).
	\end{matrix}
\end{align*}
Then $\phi=(B^*L^*LB)_2(h,0)+(B^*L^*LB)_2(0,k)$ can be equivalently expressed by $G(k)$\\$=(B^*L^*LB)_2(0,k)=\phi -(B^*L^*LB)_2(h,0)$ and (\ref{key23}) follows.
%
%
%
%
	\end{proof}
	In the following we present the injectivity results for the Newton derivative $DF_\gamma(v,c)$. The final surjectivity results and uniform boundedness of $DF_\gamma(v,c)^{-1}$ with respect to $\gamma \rightarrow 0$ and $\kappa> 0$ can be found in section \ref{SurjResultsSection}. Combined, these results will allow us to conclude, that the super linear convergence of Theorem \ref{superlin} holds for our control problem at least in the case $\kappa\neq 0$.
		\begin{theorem}\label{Theorem_injectivity_newtonDeriv}
		If $\gamma>0,\alpha_i>0$, $i=1,\cdots,m$, and $(v,c)\in L^2(I)^m\times \mathbb{R}^m$, the Newton derivative $DF_\gamma(v,c)$ is injective.
		\end{theorem}
	\begin{proof}
Case $\kappa=0$: Let $0\neq(h,k) \in L^2(I)^m\times \mathbb{R}^m$ and assume that
	$DF_\gamma(v,c)(h,k)=0$. By the first line of $DF_\gamma(v,c)(h,k)$, see Lemma \ref{Case2Df}, we then have
		\begin{align}\label{hilfAnn}
		\begin{matrix}
						h_i+X_i
						\left(\frac{1}{\gamma}  (B^*L^*LB)_{1i}(h,k)\right)=0
		\end{matrix}
		\end{align}
		for all $i=1,\cdots,m$. In the set $I_{0,i}$ it holds that
		$
		0=h_i
		$ and in
		$I_{1,i}$ we have
		$
		-\gamma h_i
		=(B^*L^*LB)_{1i}(h,k).
		$
	Furthermore, by the second row of $DF_\gamma(v,c)(h,k)$, see Lemma \ref{Case2Df}, we have
		$
		(B^*L^*LB)_{2}(h,k)=0.
		$
	Thus, we get by the positivity of $B^*L^*LB$ and (\ref{hilfAnn})
	\begin{align}\label{NoZeroSet}
	\begin{matrix}
	0\leq \left\langle B^*L^*LB(h,k),\begin{pmatrix}
	h\\k
	\end{pmatrix}\right\rangle_{L^2(I)^m\times \mathbb{R}^m}=
	-\gamma\sum\limits_{i=1}^m\int\limits_{I_{1,i}} h_i^2dt.
	\end{matrix}
	\end{align}
	This implies that $h= 0$ for all $i$. For $h=0$, we have that
$
	0=(B^*L^*LB)_{2}(0,k)
$ based on the second row of $DF_\gamma(v,c)(h,k)$.
	Because $(B^*L^*LB)_{2}(0,\cdot)$ is invertible, the kernel is ${0}$ and thus $k=0$, which is a contradiction. Hence $DF_\gamma(v,c)$ is injective.

	Case $\kappa> 0$: Let $(h,k)\neq 0 \in L^2(I)^m\times \mathbb{R}^m$ and assume that
			$DF_{\gamma,\kappa}(v,c)(h,k)=0$. By the first row of $DF_\gamma(v,c)(h,k)$, see Lemma \ref{Case2Df}, we then have
			\begin{align}
			\begin{matrix}
			h_i+X_i
			\left(\frac{1}{\gamma}  (B^*L^*LB)_{1i}(h,k)\right)=0
			\end{matrix}
			\end{align}
			for all $i=1,\cdots,m$. In the set $I_{0,i}$ we have
			$
			0=h_i,
			$
			and in $I_{1,i}$ we have
			$
			-\gamma h_i
			=(B^*L^*LB)_{1i}(h,k).
			$
			By the second row of $DF_\gamma(v,c)(h,k)$, see Lemma \ref{Case2Df}, we have
			$
			-\frac{\kappa(\gamma)}{\gamma} k=(B^*L^*LB)_{2}(h,k).
			$
			Thus, we get by the positivity of $B^*L^*LB$:
			\begin{align}\label{Contradiction_argu_kappa_no_0}
			\begin{matrix}
			0\leq \left\langle B^*L^*LB(h,k),\begin{pmatrix}
			h\\k
			\end{pmatrix}\right\rangle_{L^2(I)^m\times \mathbb{R}^m}= -\gamma\sum\limits_{i=1}^m\int\limits_{I_{1,i}} h_i^2dt -\frac{\kappa(\gamma)}{\gamma} \|k\|^2_{\mathbb{R}^m}
			\underbrace{<}_{(h,k)\neq 0}0,
			\end{matrix}
			\end{align}
			 which is a contradiction. Hence $DF_\gamma(v,c)$ is injective.
		\end{proof}
\subsection{Surjectivity Results for the Newton Derivative $DF_\gamma$}\label{surSection}
In this section we present surjectivity results for the Newton derivative $DF_{\gamma,\kappa}(v,c)$ as well as uniform boundedness of the operator family $\lbrace DF_{\gamma}(v,c)^{-1}\rbrace_{(v,c)\in L^2(I)^m\times \mathbb{R}^m}$.
\label{SurjResultsSection}
	\begin{theorem}\textcolor{white}{text}\\\label{surjectivity_proof_ofNewTon}
			For $\gamma,\kappa(\gamma)$, and $\alpha_i$, $i=1,\cdots,m$, all positive, the Newton derivative $DF_{\gamma
			}(v,c)$ is surjective for each $(v,c)\in L^2(I)^m\times \mathbb{R}^m$. Furthermore, the operator family $\lbrace DF_{\gamma}(v,c)^{-1}\rbrace_{(v,c)\in L^2(I)^m\times \mathbb{R}^m}$ is uniformly bounded for each fixed $\kappa> 0$ in $\mathfrak{L}(L^2(I)^m\times \mathbb{R}^m)$.
		\end{theorem}

	\begin{proof}\textcolor{white}{text}

	(i) Surjectivity:
		In the following consider $(v,c)$ and $(\phi_1,\phi_2)\in L^2(I)^m\times \mathbb{R}^m$. We have to show that there exists a $(h,k)\in L^2(I)^m\times \mathbb{R}^m$ such that
		\begin{align}\label{13label27}
			\begin{matrix}
			DF_{\gamma}(v,c)\begin{pmatrix}
			h\\ k
			\end{pmatrix}=\begin{pmatrix}
			\phi_1\\\phi_2
			\end{pmatrix}.
			\end{matrix}
		\end{align}
		In view of (\ref{13label27}) and (\ref{newProof}) we have $h_i=\phi_{1,i}$ for $i=1,\cdots,m$ in $I_{0,i}$. This implies that $h_i=\phi_{1,i} \mathbf{1}_{I_{0,i}}+\tilde{h}_i\mathbf{1}_{I_{1,i}}$ for some $\tilde{h}_i$, $i=1,\cdots,m$. If it holds that $|I_{1,i}|=0$, for all $i=1,\cdots,m$, we get that $h_i=\phi_{1,i}$ and by (\ref{13label27}) we have
		\begin{align}\label{13label277}
		\begin{matrix}
		\phi_2-\frac{1}{\gamma}(B^*L^*LB)_2(\phi_1,0)=\frac{\kappa(\gamma)}{\gamma} k + (B^*L^*LB)_2(0,k)=:\tilde{W}(k) \text{ with }\tilde{W}\in \mathbb{R}^{m\times m}.
		\end{matrix}
		\end{align}
		Since, $k\mapsto (B^*L^*LB)_2(0,k)$ is self-adjoint and positive definite from $\mathbb{R}^m$ to itself, there exists $k\in \mathbb{R}^m$ which solves (\ref{13label277}).

		Next, w.l.o.g. let $|I_{1,i}|,\cdots,|I_{1,\tilde{n}}|>0$ and $|I_{1,\tilde{n}+1}|,\cdots,|I_{1,m}|=0$ with $\tilde{n}>0$. By (\ref{13label27}), in $I_{1,i}$, $i=1,\cdots, \tilde{n}$, we require
				\begin{align}\label{15label27}
				\begin{matrix}
				\phi_{1,i}=\tilde{h}_i\mathbf{1}_{I_{1,i}}+\frac{1}{\gamma}(B^*L^*LB)_{1,i}(h,k)\\ \\
				\Leftrightarrow\\ \\
				\phi_{1,i}\mathbf{1}_{I_{1,i}}-\frac{1}{\gamma}(B^*L^*LB)_{1,i}(\left(\phi_{1,i} \mathbf{1}_{I_{0,i}}\right)_{i=1}^m,0)\mathbf{1}_{I_{1,i}}=\tilde{h}_i\mathbf{1}_{I_{1,i}}+\frac{1}{\gamma}(B^*L^*LB)_{1,i}((\tilde{h}_i\mathbf{1}_{I_{1,i}})_{i=1}^m,k)\mathbf{1}_{I_{1,i}}.
				\end{matrix}
				\end{align}
where in the last step we used $h_i=\phi_{1,i}\mathbf{1}_{I_{1,i}}+\tilde{h}_i\mathbf{1}_{I_{0,i}}$.
Note that $\tilde{h}_i=0$ for $i>\tilde{n}$ holds. By the second equation in (\ref{13label27}) we have to fulfill the equation
\begin{align}\label{15label277}
\begin{matrix}
\phi_2-\frac{1}{\gamma}(B^*L^*LB)_2((\phi_{1,i}\mathbf{1}_{I_{1,i}})_{i=1}^m,0)=\frac{\kappa(\gamma)}{\gamma}k+\frac{1}{\gamma}(B^*L^*LB)_2(((\tilde{h}\mathbf{1}_{I_{1,i}})_{i=1}^{\tilde{n}},0_{m-\tilde{n}}),k).
\end{matrix}
\end{align}
with $\vert I_{0,i}\vert=|I|$ for $i=\tilde{n}+1,\cdots,m$.
We will use the solution of (\ref{15label27}) and (\ref{15label277}) to obtain the solution for (\ref{13label27}).
Let us consider in the following the linear continuous and self-adjoint operator
		\begin{align*}
		\begin{matrix}
		W_1=\begin{pmatrix}
		id_{\prod\limits_{i=1}^{\tilde{n}}L^2(I_{1,i})} & 0\\ \\
		0 & \frac{\kappa(\gamma)}{\gamma}id_{\mathbb{R}^m}
		\end{pmatrix}+\frac{1}{\gamma}\underbrace{\begin{pmatrix}
		\left[B^*L^*LB_1(\left((\cdot \mathbf{1}_{I_{1,i}}\right)_{i=1}^{\tilde{n}},0_{m-{\tilde{n}}}),\cdot)\right]_{i=1}^{\tilde{n}}
		\\B^*L^*LB_2(\left((\cdot \mathbf{1}_{I_{1,i}}\right)_{i=1}^{\tilde{n}},0_{m-{\tilde{n}}}),\cdot)
		\end{pmatrix}}_{:=W_2}\\ \\
		\in \mathfrak{L}\left(\prod\limits_{i=1}^{\tilde{n}}L^2(I_{1,i})\times \mathbb{R}^m, \prod\limits_{i=1}^{\tilde{n}}L^2(I_{1,i})\times \mathbb{R}^m\right) \text{ resulting in }\\ \\
		W_1\begin{pmatrix}
		a\\ b
		\end{pmatrix}= \begin{pmatrix}
		a\\\frac{\kappa(\gamma)}{\gamma}b
		\end{pmatrix}+\frac{1}{\gamma}\begin{pmatrix}
		\left[B^*L^*LB_1(\left((a_i \mathbf{1}_{I_{1,i}}\right)_{i=1}^{\tilde{n}},0_{m-{\tilde{n}}}),b)\right]_{i=1}^{\tilde{n}}\\
		B^*L^*LB_2(\left((a_i \mathbf{1}_{I_{1,i}}\right)_{i=1}^{\tilde{n}},0_{m-{\tilde{n}}}),b)
		\end{pmatrix}
		\end{matrix}
		\end{align*}
		Since $W_2$ is non-negative, we conclude that $W_1$ is positive definite and hence invertible.
This implies that there exists a $(\tilde{h},k)\in \prod\limits_{i=1}^{\tilde{n}}L^2(I_{1,i}) \times \mathbb{R}^m$ such that (\ref{15label27}) and (\ref{15label277}) holds true. Defining $h_i=\phi_{1,i}\mathbf{1}_{I_{0,i}}+\tilde{h}_i\mathbf{1}_{I_{1,i}}$, for $i=1,\cdots,\tilde{n}$, and $h_i=\phi_{1,i}$ with $\tilde{h}_i=0$, for $i=\tilde{n}+1,\cdots,m$ provides the desired solution $(h,k)\in L^2(I)^m\times\mathbb{R}^m$ for (\ref{13label27}): In fact, for $i=\tilde{n}+1,\cdots,m$ we have $h_i=\phi_{1,i}$ a.e. in $I$. In $I_{0,i}$, $i=1,\cdots,m$, holds $h_i=\phi_{1,i}$. Furthermore, we have in $I_{1,i}$
\begin{align*}
\begin{matrix}
h_i + \frac{1}{\gamma} (B^*L^*LB)_{1,i}(h,k)=\begin{Bmatrix} \tilde{h}_i\mathbf{1}_{I_{1,i}} +\frac{1}{\gamma}(B^*L^*LB)_{1,i}(\left(\phi_{1,i} \mathbf{1}_{I_{0,i}}\right)_{i=1}^m,0)\mathbf{1}_{I_{1,i}}\\+\frac{1}{\gamma}(B^*L^*LB)_{1,i}((\tilde{h}_i\mathbf{1}_{I_{1,i}})_{i=1}^m,k)\mathbf{1}_{I_{1,i}}\end{Bmatrix}
\underbrace{=}_{(\ref{15label27})}\phi_{1,i} \mathbf{1}_{I_{1,i}}
\end{matrix}
\end{align*}
for $i=1,\cdots,\tilde{n}$. With (\ref{15label277}) we finally have (\ref{13label27}). Hence, we have proved the surjectivity of $DF_\gamma(v,c)$.

(i) Boundedness:
Let us now show that the operator family $\lbrace DF_{\gamma}(v,c)^{-1}\rbrace_{v,c}$ is uniformly bounded for fixed $\kappa >0$.
%
%
%
Let us consider $(\phi_1,\phi_2)^T\in L^2(I)^m\times\mathbb{R}^m$. By the surjectivity and injectivity of $DF_\gamma(v,c)$, there exists a unique $(h,k)\in L^2(I)^m\times\mathbb{R}^m$ with $h_i=\phi_1\mathbf{1}_{I_{1,i}}+\tilde{h}\mathbf{1}_{I_{1,i}}$, as we used above, such that
\begin{align}\label{tilde_M_marker}
\begin{matrix}
\begin{pmatrix}
\psi_{1}\\ \psi_2
\end{pmatrix}
 = \begin{pmatrix}
\left(\tilde{h}_i\mathbf{1}_{I_{1i}}\right)_{i=1}^m \\
\frac{\kappa(\gamma)}{\gamma}k
\end{pmatrix}
+\frac{1}{\gamma}\begin{pmatrix}
\left(B^*L^*LB_{1i}((\tilde{h}_j\mathbf{1}_{I_{1j}})_{j=1}^m,k)\cdot\mathbf{1}_{I_{1i}}\right)_{i=1}^m\\
B^*L^*LB_{2}((\tilde{h}_j\mathbf{1}_{I_{1j}})_{j=1}^m,k)
\end{pmatrix}
\end{matrix}
\end{align}
with
$
\psi_1:=\left(\phi_{1i}\mathbf{1}_{I_{1i}}-\frac{1}{\gamma}B^*L^*LB_{1i}((\phi_{1j}\mathbf{1}_{I_{0j}})_{j=1}^m,0)\cdot\mathbf{1}_{I_{1i}}\right)_{i=1}^m$ and $
\psi_2$\\
$:=\phi_2-\frac{1}{\gamma}B^*L^*LB_{2}((\phi_{1j}\mathbf{1}_{I_{0j}})_{j=1}^m,0)$, compare (\ref{15label27}), and (\ref{15label277}).
 Similarly as before, assume at first that $|I_{1,1}|,\cdots,|I_{1,m}|=0$. We have $h=\phi_1$ and
\begin{align}\label{Tilde_star_far_23}
\begin{matrix}
\psi_2=\frac{\kappa(\gamma)}{\gamma}k-\frac{1}{\gamma}B^*L^*LB_{2}(0,k)=\tilde{W}(k)
\end{matrix}
\end{align}
by (\ref{13label27}), (\ref{13label277}). Recall that $\tilde{W}$ is a self-adjoint, and positive definite. Using $\tilde{W}^{-1}$ on both sides of (\ref{Tilde_star_far_23}) gives us the following:
\begin{align}\label{Thee_ineq_23}
\begin{matrix}
\|k\|_{\mathbb{R}^m}=\|\tilde{W}^{-1}(\psi_2)\|\underbrace{\leq}_{(*)} \|\tilde{W}^{-1}\| \left[
\|\phi_2\|_{\mathbb{R}^m}+\frac{1}{\gamma}\|B^*L^*LB\|\|(\phi_1,0)\|_{L^2(I)^m\times \mathbb{R}^m}
\right]\\ \\
\leq 2 \|\tilde{W}^{-1}\|  \max\left(
1,\frac{1}{\gamma}|B^*L^*LB\|
\right)\|(\phi_1,\phi_2)\|_{L^2(I)^m\times \mathbb{R}^m}
\end{matrix}
\end{align}
where in $(*)$ we used that $\begin{matrix}
\|(B^*L^*LB)_2((\phi_{1j}\mathbf{1}_{I_{0j}})_{j=1}^m,0)\|_{\mathbb{R}^m}
\leq \|B^*L^*LB((\phi_{1j}\mathbf{1}_{I_{0j}})_{j=1}^m,0)\|_{L^2(I)^m\times \mathbb{R}^m}
\end{matrix}$. Note that $\tilde{W}$ is independent of $(v,c)$, $(h,k)$, and $(\phi_1,\phi_2)$. Furthermore, we have the following:
\begin{align*}
\begin{matrix}
	\left\| \begin{pmatrix}
	h\\k
	\end{pmatrix} \right\|_{L^2(I)^m\times\mathbb{R}^m}^2 \underbrace{=}_{h=\phi_1}
	\left\langle
	\begin{pmatrix}
	\phi_1\\ k
	\end{pmatrix},	
	\begin{pmatrix}
	\phi_1\\ k
	\end{pmatrix}
	\right\rangle_{L^2(I)^m\times\mathbb{R}^m}
	\leq \|(\phi_1,\phi_2)\|_{L^2(I)^m\times \mathbb{R}^m}^2 +\|k\|_{\mathbb{R}^m}^2
	\\ \\ \underbrace{\leq}_{(\ref{Thee_ineq_23})}
	\|(\phi_1,\phi_2)\|_{L^2(I)^m\times \mathbb{R}^m}^2 +4\|\tilde{W}^{-1}\|^2\max(1,\frac{1}{\gamma} \|B^*L^*LB\|)^2
	\|(\phi_1,\phi_2)\|_{L^2(I)^m\times\mathbb{R}^m}^2\\
	\leq \tilde{\tilde{c}}\|(\phi_1,\phi_2)\|_{L^2(I)^m\times\mathbb{R}^m}^2
\end{matrix}
\end{align*}
with $\tilde{\tilde{c}}>0$ independent of $(v,c)$, $(h,k)$, and $(\phi_1,\phi_2)$. Hence, we have
\begin{align*}
	\begin{matrix}
	\|DF_\gamma(v,c)^{-1}\|
	=\sup\limits_{\begin{Bmatrix}\|(\phi_1,\phi_2)\|_{L^2(I)^m\times\mathbb{R}^m}\leq 1\\
	(\phi_1,\phi_2)^T=DF_\gamma(v,c)(h,k)	
	\end{Bmatrix}}\|(h,k)\|_{L^2(I)^m\times\mathbb{R}^m}\leq \tilde{\tilde{c}} <\infty.
	\end{matrix}
\end{align*}
Next assume again that $|I_{1,1}|,\cdots,|I_{1,\tilde{n}}|>0$, and $|I_{1,\tilde{n}+1}|,\cdots,|I_{1,m}|=0$ with $\tilde{n}>0$. Let us define $\tilde{\tilde{h}}:=\left(\tilde{h}_j \mathbf{1}_{I_{1j}}\right)_{j=1}^m\in L^2(I)^m$, $\mathcal{N}:=\lbrace 1,\cdots,\tilde{n}\rbrace$,
 $\prod\limits_{\mathcal{N}}:=\prod\limits_{i=1}^{\tilde{n}}L^2(I_{1,i})$,
 and $\tilde{h}_{\mathcal{N}}:=\left(\tilde{h}_j \mathbf{1}_{I_{1j}}\right)_{j=1}^{\tilde{n}}\in \prod\limits_{\mathcal{N}}$. We have $h_i=\phi_{1,i}$ for all $i=\tilde{n}+1,\cdots,m$ and
\begin{align}\label{3_vertical_1_horizontal_hash_sign}
\begin{matrix}
\psi_2=\frac{\kappa(\gamma)}{\gamma}k-\frac{1}{\gamma}B^*L^*LB_{2}(\tilde{\tilde{h}},k).
\end{matrix}
\end{align}
By (\ref{tilde_M_marker}) we have
\begin{align}\label{3_horizontal_2_vertical_hash_sign}
\begin{matrix}
\left\langle
\begin{pmatrix}
\left(
\psi_{1i}
\right)_{i\in \mathcal{N}}\\
\psi_2
\end{pmatrix},\begin{pmatrix}
\tilde{h}_{\mathcal{N}}\\
k
\end{pmatrix}
\right\rangle_{\prod\limits_{\mathcal{N}}\times \mathbb{R}^m}=\begin{Bmatrix*}[l]
\left\langle
\begin{pmatrix}
\tilde{h}_{\mathcal{N}}\\
\frac{\kappa(\gamma)}{\gamma}k
\end{pmatrix},\begin{pmatrix}
\tilde{h}_{\mathcal{N}}\\
k
\end{pmatrix}
\right\rangle_{\prod\limits_{\mathcal{N}}\times \mathbb{R}^m}\\ \\
+\frac{1}{\gamma}
\left\langle
\begin{pmatrix}
\left(
B^*L^*LB_{1i}(\tilde{\tilde{h}},k)\mathbf{1}_{I_{1i}}
\right)_{i\in \mathcal{N}}\\
(B^*L^*LB)_2(\tilde{\tilde{h}},k)
\end{pmatrix},\begin{pmatrix}
\tilde{h}_{\mathcal{N}}\\
k
\end{pmatrix}
\right\rangle_{\prod\limits_{\mathcal{N}}\times \mathbb{R}^m}
\end{Bmatrix*}.
\end{matrix}
\end{align}
This equation implies the following:
\begin{align*}
\begin{matrix}
\|(\psi_1,\psi_2)\|_{L^2(I)^m\times \mathbb{R}^m}\|(\tilde{h}_{\mathcal{N}},k)\|_{\prod\limits_{\mathcal{N}}\times\mathbb{R}^m}\geq \|\left((\psi_{1i})_{i\in \mathcal{N}},\psi_2\right)\|_{\prod\limits_{\mathcal{N}}\times\mathbb{R}^m}\|(\tilde{h}_{\mathcal{N}},k)\|_{\prod\limits_{\mathcal{N}}\times\mathbb{R}^m}\\ \\
\geq \left\langle
\begin{pmatrix}
\left(
\psi_{1i}
\right)_{i\in \mathcal{N}}\\
\psi_2
\end{pmatrix},\begin{pmatrix}
\tilde{h}_{\mathcal{N}}\\
k
\end{pmatrix}
\right\rangle_{\prod\limits_{\mathcal{N}}\times \mathbb{R}^m}\underbrace{\geq}_{(\ref{3_horizontal_2_vertical_hash_sign})}
\begin{Bmatrix*}[l]
\min\left(
1,\frac{\kappa(\gamma)}{\gamma}
\right)\|(\tilde{h}_{\mathcal{N}},k)\|_{\prod\limits_{\mathcal{N}}\times\mathbb{R}^m}^2\\ \\
+\frac{1}{\gamma}
\left\langle
\begin{pmatrix}
\left(
B^*L^*LB_{1i}(\tilde{\tilde{h}},k)\mathbf{1}_{I_{1i}}
\right)_{i\in \mathcal{N}}\\
(B^*L^*LB)_2(\tilde{\tilde{h}},k)
\end{pmatrix},\begin{pmatrix}
\tilde{h}_{\mathcal{N}}\\
k
\end{pmatrix}
\right\rangle_{\prod\limits_{\mathcal{N}}\times \mathbb{R}^m}
\end{Bmatrix*}\\ \\
\underbrace{\geq}_{(**)}\min\left(
1,\frac{\kappa(\gamma)}{\gamma}
\right)\|(\tilde{h}_{\mathcal{N}},k)\|_{\prod\limits_{\mathcal{N}}\times\mathbb{R}^m}^2
\end{matrix}
\end{align*}
where (**) follows by the non-negativity of $B^*L^*LB$, i.e.
\begin{align*}
\begin{matrix}
\frac{1}{\gamma}
\left\langle
\begin{pmatrix}
\left(
B^*L^*LB_{1i}(\tilde{\tilde{h}},k)\mathbf{1}_{I_{1i}}
\right)_{i\in \mathcal{N}}\\
(B^*L^*LB)_2(\tilde{\tilde{h}},k)
\end{pmatrix},\begin{pmatrix}
\tilde{h}_{\mathcal{N}}\\
k
\end{pmatrix}
\right\rangle_{\prod\limits_{\mathcal{N}}\times \mathbb{R}^m}=\frac{1}{\gamma}
\left\langle
B^*L^*LB(\tilde{\tilde{h}},k)
,\begin{pmatrix}
\tilde{\tilde{h}}\\
k
\end{pmatrix}
\right\rangle_{L^2(I)^m\times \mathbb{R}^m}\geq 0,
\end{matrix}
\end{align*}
where we used that $\tilde{\tilde{h}}_j=\tilde{h}_j\mathbf{1}_{I_{1,j}}=\tilde{h}_{\mathcal{N},j}$ for $j=1,\cdots,\tilde{n}$ and $\tilde{\tilde{h}}_j=0$ for $j>\tilde{n}$. Hence, we have
\begin{align}\label{Another_impor_ineq232}
\begin{matrix}
\frac{1}{\min\left(
1,\frac{\kappa(\gamma)}{\gamma}
\right)}\|(\psi_1,\psi_2)\|_{L^2(I)^m\times \mathbb{R}^m}\geq \|(\tilde{h}_{\mathcal{N}},k)\|_{\prod\limits_{\mathcal{N}}\times\mathbb{R}^m}.
\end{matrix}
\end{align}
Finally, we have by
(\ref{Another_impor_ineq232}) and the definition of $\psi_i$
\begin{align*}
\begin{matrix}
\|(h,k)\|_{L^2(I)^m\times \mathbb{R}^m}^2=
\|((\phi_{1j}\mathbf{1}_{I_{0j}})_{i=1}^m+\tilde{\tilde{h}},k)\|_{L^2(I)^m\times \mathbb{R}^m}^2\\ \\
=\left\langle
\begin{pmatrix}
(\phi_{1j}\mathbf{1}_{I_{0j}})_{i=1}^m+\tilde{\tilde{h}} \\ k
\end{pmatrix},
\begin{pmatrix}
(\phi_{1j}\mathbf{1}_{I_{0j}})_{i=1}^m+\tilde{\tilde{h}} \\ k
\end{pmatrix}
\right\rangle_{L^2(I)^m\times \mathbb{R}^m}\\ \\
\underbrace{=}_{\mathbf{1}_{I_{1i}}\cdot\mathbf{1}_{I_{0i}}=0} \|(\phi_{1j}\mathbf{1}_{I_{0j}})_{i=1}^m\|_{L^2(I)^m}^2+\|(\tilde{\tilde{h}},k)\|_{L^2(I)^m\times \mathbb{R}^m}^2
\leq \|(\phi_1,\phi_2)\|_{L^2(I)^m\times \mathbb{R}^m}^2+\|(\tilde{h}_{\mathcal{N}},k)\|_{\prod\limits_{\mathcal{N}}\times \mathbb{R}^m}^2\\ \\
\underbrace{\leq}_{(\ref{Another_impor_ineq232})}\|(\phi_1,\phi_2)\|_{L^2(I)^m\times \mathbb{R}^m}^2 + \frac{1}{\min\left(
1,\frac{\kappa(\gamma)}{\gamma}
\right)^2}\|(\psi_1,\psi_2)\|_{L^2(I)^m\times \mathbb{R}^m}^2\leq \tilde{\tilde{c}}^2\|(\phi_1,\phi_2)\|_{L^2(I)^m\times \mathbb{R}^m}^2
\end{matrix}
\end{align*}
where $\tilde{\tilde{c}}>0$ is some constant independent of $(v,c)$, and $(h,k)$. This finally concludes the boundedness of $\|DF_\gamma(v,c)^{-1}\|\leq \tilde{\tilde{c}}$.
\end{proof}

As a consequence of Theorem \ref{superlin} - \ref{surjectivity_proof_ofNewTon} we have the following result.
\begin{corollary}
	If $\gamma$, $\kappa(\gamma)$, and $i\in\lbrace 1,\cdots,m\rbrace$ are all positive, then the semi-smooth Newton algorithm 	
	\begin{align}
	\begin{matrix}
	(v^{k+1},c^{k+1})=(v^k,c^k)-DF_{\gamma}(v^k,c^k)^{-1}F_\gamma(v^k,c^k), 
	\end{matrix}
	\end{align}
	converges super-linearly to the optimal solution $(\overrightarrow{v},\overrightarrow{c})$ of $(\tilde{P}_\gamma)$, provided that\\ $\|(v^0,c^0) - (\overrightarrow{v},\overrightarrow{c})\|_{L^2(I)^m\times\mathbb{R}^m}$ is sufficiently small.
\end{corollary}

\section{Numerics and Examples}
\label{numericsection}
	In the following sections we present numerical results which illustrate the effect of $BV$ cost on the optimal controls. For the discretization of $(\mathcal{W})$ we used the 3-level finite element method presented in \cite{[Zlot]}. In particular, we used the Crank-Nicholson method with linear continuous finite elements in time ($S_\tau$) and space ($S_h$). The resulting discrete solution of $(\mathcal{W})$ is an element in the tensor space $S_\tau \otimes S_h $.

	We discretized the control $(v,c)\in L^2(I)^m\times \mathbb{R}^m$ in $(\tilde{P}_\gamma)$ by $S_\tau$ elements. Furthermore, we used the trapezoidal rule to evaluate all time-depending integrals in problem $(\tilde{P}_\gamma)$. The trapezoidal rule guarantees that the  function inside the prox operator (see (\ref{proxerop})) attains its maximum and minimum in the time nodes we considered for $S_\tau$. We used the mass matrix for the space depending integral in $(\tilde{P}_\gamma)$ with respect to the finite elements in $S_h$. Further details can be found in \cite{PhDEngel}.

In the following sections, we construct two test cases in such a manner that exact analytic solutions for $(P)$, respectively $(\tilde{P})$, become available. We use Algorithm \ref{Algo1}, which is a BV-path following algorithm, to approximate numerically the solutions of those examples. The solution of the linear system in Algorithm \ref{Algo1} is approximated by a Krylov iterative method.

A similar path following algorithm is used in the semi-linear parabolic case in \cite{[KCK]}.
\begin{algorithm}
	Input: $(v_0,c_0)\in L^2(I)^m\times\mathbb{R}^m$, $\gamma_0>0$, $TOL_{\gamma}>0$, $TOL_{N}>0$, $k=0$ and $\nu\in (0,1)$\\
	\While{$\gamma_k>TOL_{\gamma}$}{
		Set $i=0$, $(v_k^i,c_k^i)=(v_k,c_k)$\\
		\While{$\|F_{\gamma_k}(v_k^i,c_k^i)\|_{L^2(I)^m\times\mathbb{R}^m}>TOL_{N}$}{
			Solve $DF_{\gamma_k}(v_k^i,c_k^i)(\delta v,\delta c)=-F_{\gamma_k}(v_k,c_k)$, set $(v_{k+1}^{i+1},c_{k+1}^{i+1})=(v_k^i,c_k^i) + (\delta v,\delta c)$; $i=i+1$.
		}
		Define $(v_{k+1},c_{k+1})=(v_k^i,c_k^i)$ and $\gamma_{k+1}=\nu \gamma_k$; set $k=k+1$.
	}
	\caption{BV-path following algorithm.}\label{Algo1}
\end{algorithm}
A special aspect about the semi-smooth Newton method inside the BV-Path-Following algorithm compared to the one in \cite{[KCK]} is that we consider the derivative and an additional constant as control instead of a BV function. Besides, we have an additional term $\frac{\kappa(\gamma)}{\gamma}c$, which allows us to obtain super-linear convergence for the semi-smooth Newton algorithm for $\kappa\neq 0$, see section \ref{SuperLinSection} and \ref{surSection}.

\subsection{Construction of Test Examples}\label{GeneralExample}
This example is constructed in such a way that for the optimal adjoint state a wide range of different sets $\lbrace t\in I \vert \vert p_1(t)\vert=\alpha\rbrace$ are possible. For the construction of test examples we consider $\Omega\subset \mathbb{R}^d$, $d\in\lbrace 1,2,3\rbrace$, $I=(0,T)$ with $T<\infty$, $m\in \mathbb{N}_{>0}$, $\alpha_i>0$, and $g_i\in L^\infty(\Omega)\setminus \lbrace 0 \rbrace$ with disjoint supports $\supp(g_i)=w_i$ for $i=1,\cdots,m$.

$\textbf{Adjoint State:}$
 Choose a function $f\in H^2(\Omega)\cap V$ with $\int\limits_{\Omega}f\cdot g_i dx\neq 0$ for all $i=1,\cdots,m$. Furthermore, let $h \in H^2(I)$ such that $h(T)=0$, $\partial_t h(T)=0$, $\int\limits_0^T h(s)ds=0$, $\tilde{p}_i(t):=\int_t^T h(s)ds \in [-\frac{\alpha_i}{\vert \beta_i \vert},\frac{\alpha_i}{\vert \beta_i \vert}]$, with $\beta_i:= \int\limits_\Omega f(x) g_i(x)dx$. Define the function
$
\overline{\varphi}(t,x):=h(t)\cdot f(x).
$
The properties of $h$ and $f$ imply that $\overline{\varphi}(T)=\partial_t \overline{\varphi}(T)=0$, and $\overline{\varphi}\vert_{\partial \Omega}=0$. Applying $\partial_{tt} - \bigtriangleup$ to $\overline{\varphi}(t,x)$ implies
$
(\partial_{tt} - \bigtriangleup)\overline{\varphi}=f\cdot\partial_{tt}h - h\cdot\bigtriangleup f
$
a.e. in $\Omega_T$.
Furthermore, we have
\begin{align*}
\begin{matrix}
p_{1,i}(t):=\int\limits_{t}^T \int\limits_\Omega \overline{\varphi}(s,x)  g_i(x)dxds= \int\limits_{t}^T h(s)ds
\int\limits_\Omega f(x) g_i(x)dx=\beta_i \cdot \tilde{p}_i(t),
\end{matrix}
\end{align*}
with $\|p_{1,i}\|_\infty \leq \alpha_i$, and $p_{1,i} \in C_0(I)$, for $i=1,\cdots,m$.

$\textbf{Control:}$
Let us consider now arbitrary positive measures $\mu_{i}^+,\mu_{i}^- \in M(I)$ with support $\supp(\mu_{i}^\pm)$ $\subset \lbrace p_{1,i} =\mp\alpha_i \rbrace$ and define
$
d D_t u_{1,i}(t):=- \frac{\alpha_i}{p_{1,i}(t)} (\mu_i^+ + \mu_i^-).
$
Due to the continuity of $p_{1,i}$ the support of $\mu_i^+$ is disjoint from the support of $\mu_i^-$. The measure $D_t u_{1,i}$ is a positive measure on $\lbrace p_{1,i} =-\alpha_i \rbrace$ and a negative measure on $\lbrace p_{1,i} =+\alpha_i \rbrace$.
On the support of $\mu_i^\pm$ we have $\vert \frac{\alpha_i}{p_{1,i}(s)}\vert =1$.
This gives us
\begin{align*}
\begin{matrix}
\int\limits_0^T - \frac{p_{1,i}(s)}{\alpha_i} dD_tu_{1,i}(s)=\int\limits_0^T d\mu_i^+(s) +\int\limits_0^Td\mu_i^-(s)\\ \\
\underbrace{=}_{\text{positive }\mu_i^\pm} \|\mu_i^+ +\mu_i^-\|_{M(I)}\underbrace{=}_{\vert \frac{\alpha_i}{p_{1,i}(s)}\vert =1} \left\|- \frac{\alpha_i}{p_{1,i}(s)} (\mu_i^+ +\mu_i^-)\right\|_{M(I)}=\|D_t u_{1,i}\|_{M(I)}.
\end{matrix}
\end{align*}
Let us define $\overline{u}_{1,i}:=u_{1,i}+c_i$ with $\overline{u}_{1,i}(t)=\int_0^t dD_t u_{1,i}(s)$ and $c_i\in \mathbb{R}$.

$\textbf{State and Desired State:}$
Furthermore, let us fix the desired state according to
$
y_d=\tilde{S}(\overline{u})-(f\cdot\partial_{tt}h - h\cdot\bigtriangleup f),
$
and some displacement and velocity functions $(y_0,y_1)\in V \times H$.
For the resulting problem $(P)$ the function $\overline{u}$ is the optimal control.

\subsection{Finitely Many Jumps Example}
	\label{Finite Many Jumps Example}
	This example is constructed in such a way that the set $\lbrace t\in I \vert \vert p_1(t)\vert=\alpha \rbrace$ consists of finitely many active points. A similar construction steps can also be found in \cite[Example 1]{[KCK]}.
	Let $\beta>0$, $l \in \mathbb{N}_{>0}$, $d\in\lbrace1,2,3\rbrace$, $\Omega=(-1,1)^d$, $I=(0,2)$, and define
	\begin{align*}
	\begin{matrix}
	g(x)=\mathds{1}_{[-0.5,0.5]^d}(x)=\prod_{i=1}^d\mathds{1}_{[-0.5,0.5]}(x_i).
	\end{matrix}
	\end{align*}
	Define the function $\overline{\varphi}(t,x)$ by
$
	\beta \sin(l\pi t)\sin(l\frac{\pi}{2} t)\prod_{i=1}^d\cos(\frac{\pi}{2}x_i).
$
	Then $\overline{\varphi}$ has the property $\overline{\varphi}\vert_{\partial \Omega}\equiv 0$, $\overline{\varphi}(2)=0$, and
	\begin{align*}
	\begin{matrix}
	\partial_t\overline{\varphi}\vert_{t=2}(t,x)=\beta\left(l\pi \cos(l\pi t)\sin(l\frac{\pi}{2} t)+l\frac{\pi}{2}\sin(l\pi t)\cos(l\frac{\pi}{2} t)\right)
	\prod_{i=1}^d\cos(\frac{\pi}{2}x_i)\vert_{t=2}=0.
	\end{matrix}
	\end{align*}
	Using the wave operator $\partial_{tt}-\bigtriangleup$ on $\overline{\varphi}(t,x)$ gives us:
	\begin{align*}
	\begin{matrix}
	(\partial_{tt}-\bigtriangleup)\overline{\varphi}(t,x)=\left(\frac{d\pi^2}{4}-\frac{5l^2\pi^2}{4}\right) \overline{\varphi}+\beta (l^2 \pi^2)\cos(l\pi t)\cos(l\frac{\pi}{2} t)\prod_{i=1}^d\cos(\frac{\pi}{2}x_i).
	\end{matrix}
	\end{align*}
	By an elementary computation we find
	\begin{align*}
	\begin{matrix}
	p_1(t):=\int\limits_{t}^2 \int\limits_\Omega \overline{\varphi}(t,x)g(x)dxdt
	=\frac{4\beta}{3\pi l}\left(-sin\left(\frac{l}{2}\pi t\right)^3\right)\left(\frac{2\sqrt{2}}{\pi}\right)^d.
	\end{matrix}
	\end{align*}
	It holds that $p_1(0)=p_1(2)=0$, and $p_1\in C_0(I)$ with $\|p_1\|_{C_0(I)}=\frac{4\beta}{3\pi l}\left(\frac{2\sqrt{2}}{\pi}\right)^d$. To have equality $\|p_1\|_{C_0(I)}=\alpha$ at the optimum, we have to chose $\beta=\alpha\frac{3\pi l}{4}\left(\frac{2\sqrt{2}}{\pi}\right)^{-d}$. Furthermore, we have
$
		\left\lbrace t\in I \left\vert
			p_1(t)=\pm \alpha
			 \right.\right\rbrace=\left\lbrace \frac{1+2n}{l}\vert n\in \lbrace 0,\cdots,l-1 \rbrace \right\rbrace \subset I.$
\begin{figure}
	 \hspace*{-3.5cm}
	 \includegraphics[width=1.5\textwidth,]{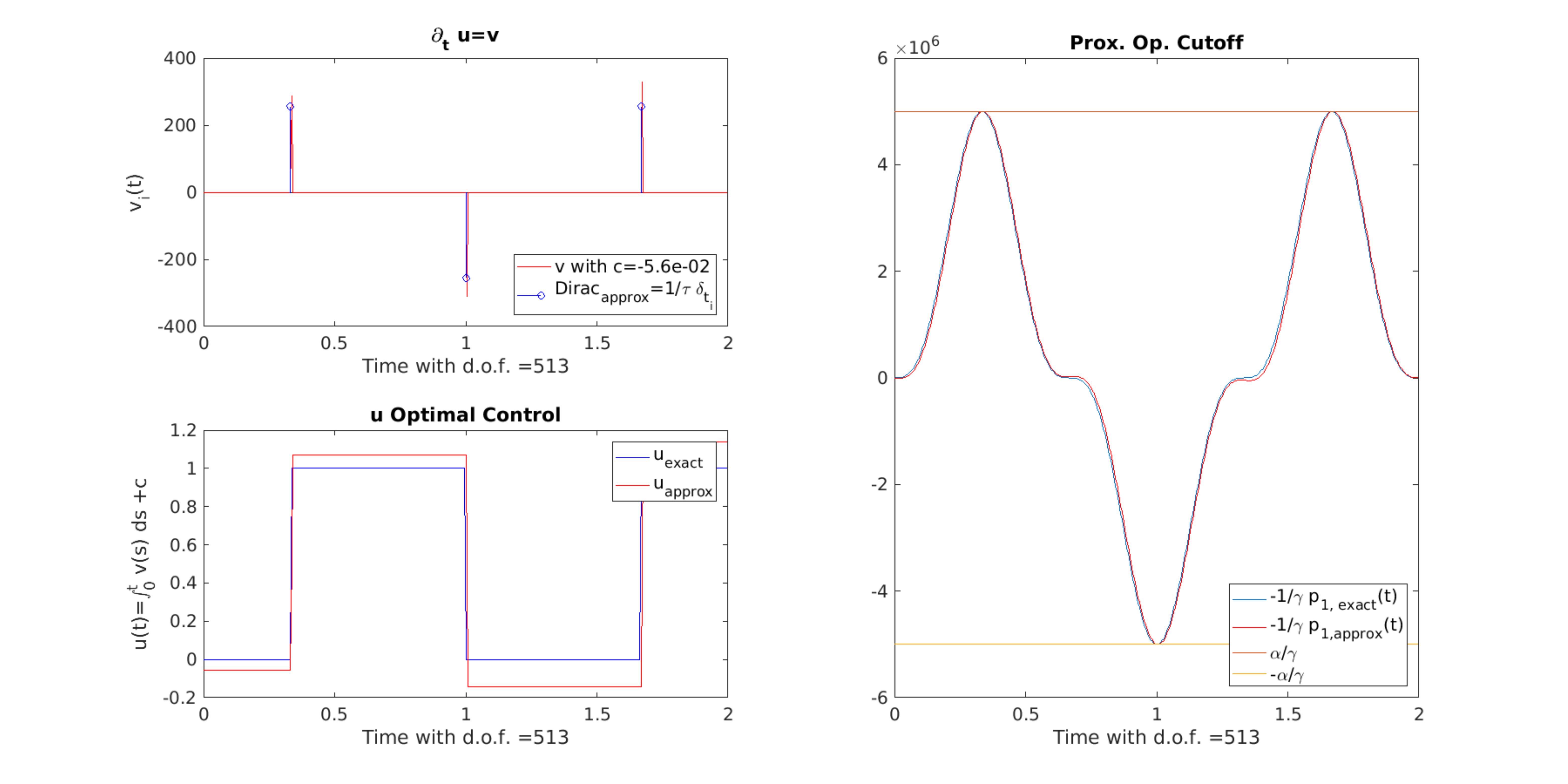}
	 \caption{ }\label{DiracExample1}
	 \hspace*{-3.5cm}
	 \includegraphics[width=1.5\textwidth]{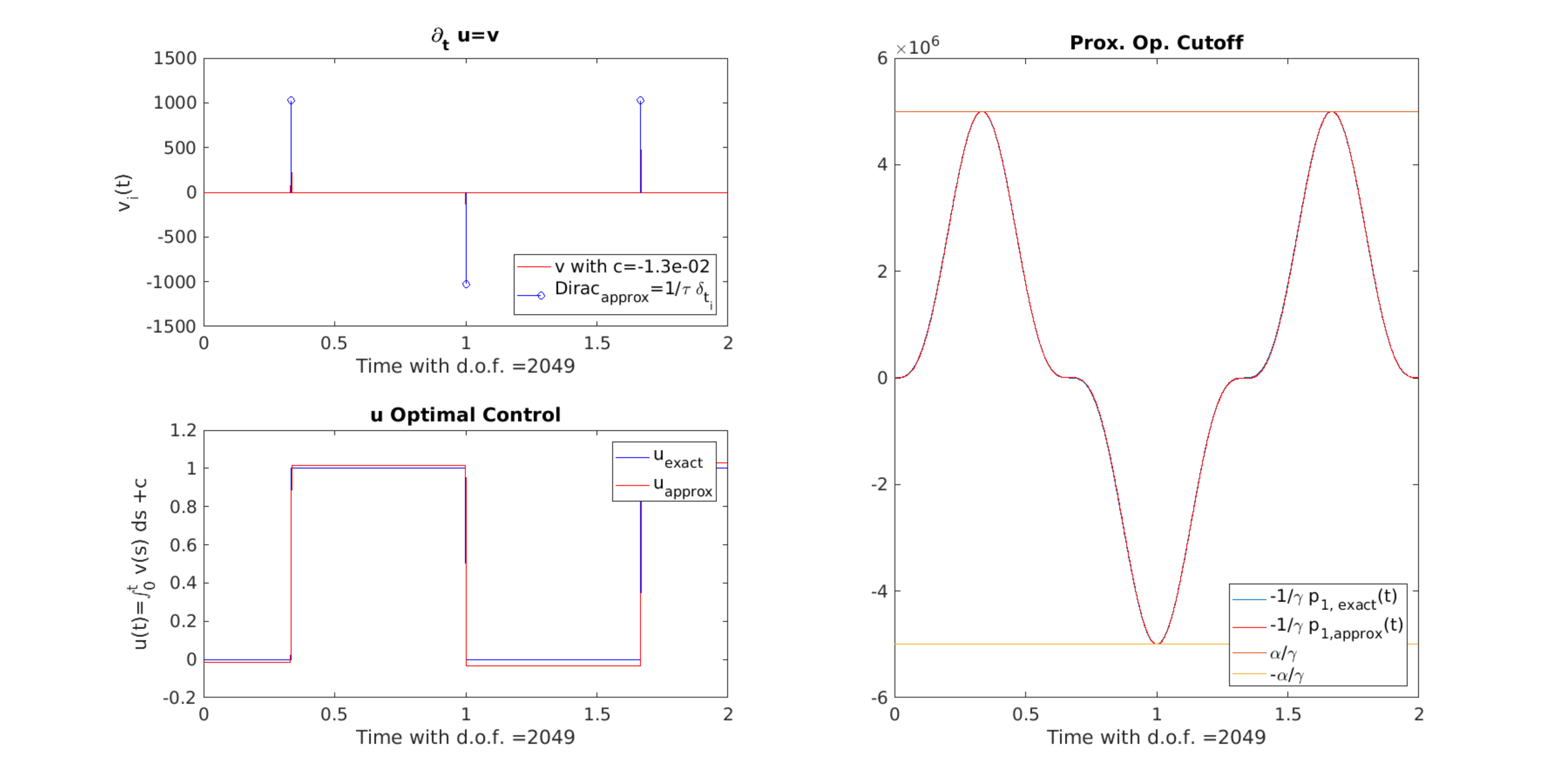}
	 \caption{ }\label{DiracExample4}
	 \end{figure}
The following equalities hold
	\begin{align}\label{differentControlsPos}
	\begin{matrix}
	-\alpha\sum\limits_{n=0}^{l-1} c_n\left(- \sin\left(\pi\frac{2n+1}{2}\right)^3\right)=\int\limits_0^2 -p_1(t)dD_t\overline{u}_1=\|p_1\|_{C_0(I)}\|D_t\overline{u}_1\|_{M(I)}=\alpha\|D_t\overline{u}_1\|_{M(I)}=\alpha\sum\limits_{n=0}^{l-1}\vert c_n\vert
	\end{matrix}
	\end{align}
	for $c_n=\sign\left( \sin\left(\pi\frac{2n+1}{2}\right) \right)$ or $0$. Consider an arbitrary $\overline{c}\in \mathbb{R}$, and define $\overline{u}:=\overline{u}_1+\overline{c}$ with
$
	 \overline{u}_1(t)=\int\limits_0^t \sum\limits_{n=0}^{l-1}c_n d\delta_{\frac{1+2n}{l}}=\sum\limits_{n=0}^{l-1}c_n \mathds{1}_{[\frac{1+2n}{l},2]}(t).
$
	 Now determine the desired state as $y_d:=\tilde{S}(\overline{u})-(\partial_{tt}-\bigtriangleup)\overline{\varphi}(t,x)$ with arbitrary $(y_0,y_1)\in V\times H$. For the resulting problem $(P)$ the function $\overline{u}$ is the optimal control. The corresponding cost functional has the value
	 \begin{align*}
	 \begin{matrix}
	 J(\overline{u})=\frac{1}{2}\|\tilde{S}(\overline{u})-y_d\|_{L^2(\Omega_T)}^2+\alpha\|D_t\overline{u}\|_{M(I)}=\frac{1}{2}\|(\partial_{tt}-\bigtriangleup)\overline{\varphi}\|_{L^2(\Omega_T)}^2+\alpha\sum\limits_{n=0}^{l-1}\vert c_n\vert\\ \\=\frac{\beta^2}{4}\left(\frac{d\pi^2}{4}-\frac{5 l^2 \pi^2}{4}\right)^2+\frac{\beta^2l^4\pi^4}{4}
	 +\alpha\sum\limits_{n=0}^{l-1}\vert c_n\vert
	 \end{matrix}
	 \end{align*}
	 where the last equality follows from an elementary computation.

	 We now turn to discuss numerical results. We considered dimension $d=2$, and the number of Diracs $l=3$. For the desired state $y_d:=\tilde{S}(\overline{u})-(\partial_{tt}-\bigtriangleup)\overline{\varphi}(t,x)$ we used $(y_0,y_1)=(0,0)$. The optimal constant is fixed by $\overline{c}=0$. The BV-path following algorithm starts with $\gamma_0=1$, $(v_0,c_0)=(0,0)$, and we iterate according to $\gamma_{k+1}=0.1\gamma_k$. We stopped the BV-path following algorithm when $\gamma_k=10^{-8}$ was reached. The function $\kappa$ is defined as $\kappa(\gamma)=\gamma^4$.

	 In the Figures \ref{DiracExample1} and \ref{DiracExample4} we depict the optimal control for two different choices of d.o.f.
	 On the right hand side of each Figure \ref{DiracExample1} and \ref{DiracExample4}, we see the function $p_{1,\text{approx}}:=\psi_1$
	 which appear in the prox operator (\ref{proxerop}). As suggested by (\ref{implicitcontrol}) we obtain $\partial_t \overline{u}_{\text{approx}}=0$ whenever $\vert p_{1,\text{approx}} \vert < \alpha$ for the derivative of the approximated optimal control.
	
	 In the upper left sub-figure in Figure \ref{DiracExample1} - \ref{DiracExample4} the  \textcolor{red}{red} curve depicts the approximated derivative of the approximated optimal control $\overline{u}_{\text{approx}}$. The blue pin line represents the exact Dirac measures approximated according to the mesh, i.e. for $a\in \mathbb{R}$ the Dirac measure $a\cdot \delta_{t}$ is approximated by a pin in the position $t$ with pin height of $\frac{a}{\tau}$ with $\tau$ the uniform distance between two time nodes. In the lower sub-figure in Figure \ref{DiracExample1} - \ref{DiracExample4} we see the exact optimal control $\overline{u}$ in \textcolor{blue}{blue}, $L^2$-projected on $V_h$, and the approximated optimal control $\overline{u}_{\text{approx}}$ in \textcolor{red}{red}.


We stopped the semi-smooth Newton algorithm as soon as $\|F_{\gamma_k}(u_k)\|_{L^2(I)^m\times\mathbb{R}^m}\leq 10^{-6}=:TOL_N$. 
In Figure \ref{DiracNewton} we show the $\|F_{\gamma_k}(u_k)\|_{L^2(I)^m\times\mathbb{R}^m}$-error for different $\gamma$ values which where used in the in Algorithm \ref{Algo1}. In Figure \ref{DiracNewton}, we see the errors which correspond $2049$ d.o.f. in time. In the last figure on the right we see the error corresponding to Figure \ref{DiracExample4}. As expected, all figures show the super-linearity of the $\|F_{\gamma_k}(u_k)\|_{L^2(I)^m\times\mathbb{R}^m}$-error.

\begin{figure}[H]
	\includegraphics[width=1.0\textwidth]{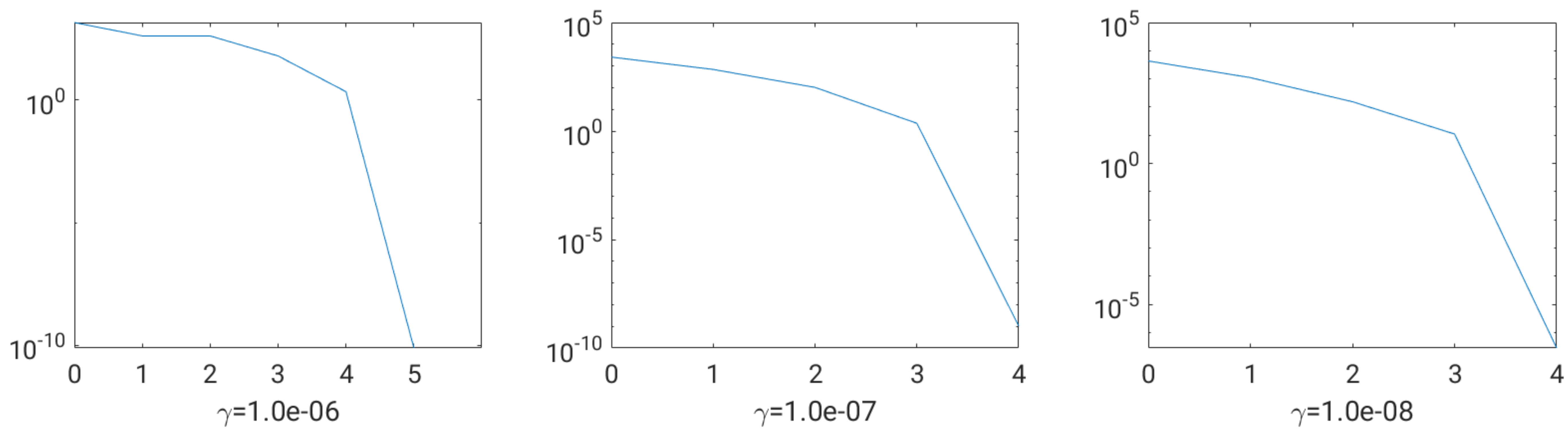}
	\caption{ }\label{DiracNewton}
\end{figure}

\subsection{Cantor Function or Devil's Staircase Example}\label{CantorSection}
Here we construct functions $p_{1,i}\in C_0(I)$ which enable us to use all three classes of measures for the distributional derivative of a $BV$ function in time. This means absolutely continuous measures with respect to the Lebesgue measure, countable linear combinations of Dirac measures, and Cantor measures. For further information about these measure characterisation see for example \cite[p. 184]{[AmFuPa]}. Finally we will use $p_{1,i}$ to create a Cantor-like optimal control.

	Let $0<a_1<b_1<a_2<b_2<T$. Then for all closed non-trivial intervals $I_i\subseteq (a_i,b_i)$, $i=1,2$, there exists $\tilde{p}\in C_c^\infty(I)$ such that $\vert \tilde{p} \vert \leq 1$ with
	\begin{align}\label{BumpingProp}
	\begin{matrix}
	\tilde{p}=\left\lbrace
	\begin{matrix}
	=0 & \text{in} & (0,T)\setminus ((a_1,b_1)\cup (a_2,b_2))\\
		\geq 0 & \text{in} & (a_1,b_1)\\
		\leq 0 & \text{in} & (a_2,b_2)\\
		1  & \text{in} & I_1\\
		-1  & \text{in} & I_2
	\end{matrix}
	\right..
	\end{matrix}
	\end{align}
	In the following we denote by $PC$ the set $\lbrace t \in I \vert  \tilde{p}(t)=\pm 1 \rbrace=I_1\cup I_2$.
Let us now fix $T,a_i,b_i$ such that the assumptions above (\ref{BumpingProp}) hold. We set $h=\partial_t \tilde{p}$, with $\tilde{p}$ as defined in (\ref{BumpingProp}). Then it holds that $h(T)=\partial_t h(T)=0$ and $\int\limits_0^T h(t)dt=\tilde{p}(T)-\tilde{p}(0)=0$ due to the compact support of $\tilde{p}$ inside $I$. Consider $\Omega$, $d$, $m$, $g_i$, $f$ such that the assumptions in section \ref{GeneralExample} "Construction of Test Examples" are fulfilled, and define $
		\overline{\varphi}(t,x):=h(t)\cdot f(x).
$
It holds that
 \begin{align*}
 \begin{matrix}
 p_{1,i}(t):=\int\limits_{t}^T \int\limits_\Omega \varphi(t,x) \cdot g_i(x)dxdt= \int\limits_{t}^T h(t)dt \cdot
\underbrace{ \int\limits_\Omega f(x) g_i(x)dx}_{:=\tilde{z}_i\neq 0}=(\tilde{p}(T)-\tilde{p}(t)) \cdot \tilde{z}_i
=-\tilde{p}(t)\cdot \tilde{z}_i.
 \end{matrix}
 \end{align*}
Under these circumstances, we define $\alpha_i:=\vert\tilde{z}_i\vert$. Now, we consider positive measures $\mu_i^\pm\in M(I)$ with support $\supp(\mu_{i}^\pm)\subset \lbrace p_{1,i} =\mp\alpha_i \rbrace$, and define
$
	d D_t u_{1,i}(t):=- \frac{\alpha_i}{p_{1,i}(t)} (\mu_i^+ + \mu_i^-).
$
Following the instructions in section \ref{GeneralExample} "Construction of Test Examples" an optimal control $\overrightarrow{u}$. The measures $\mu_i^\pm$ can be of the types described above.

Our next aim is to construct an optimal control which has a Cantor-like shape. Hence, denote by $C(t)$ the Cantor function on $[0,1]$ (see \cite[Example 3.34]{[AmFuPa]}).
Define the function $C^+(x)=C\left(\frac{x-PC^+_{left}}{2d_+}\right)$, with $d_+=\vert  I_1\vert$, $PC^+_{left}=\min_{x\in I_1}x$, on the domain $(a_1,b_1)$. Additionally, let us define $C^-(x)=C\left(\frac{PC^-_{right}-x}{2d_-}\right)$ , with $d_-=\vert  I_2\vert$, $PC^+_{right}=\max_{x\in I_2}x$, on the domain $(a_2,b_2)$. Accordingly we define the continuous function
\begin{align*}
\begin{matrix}
u_i(t):=\int\limits_0^t d(\mu_i^+(s)+ \mu_i^-):=\left\lbrace\begin{matrix}
C^+(x) & \text{, on } & I_1\\ \\
\frac{1}{2} & \text{, on }& [PC^+_{left},PC^-_{right}]\setminus PC\\ \\
C^-(x) & \text{, on } & I_2\\ \\
0 & \text{else.} &
\end{matrix}\right.
\end{matrix}
\end{align*}
\begin{figure}[H]
	 \includegraphics[width=1.0\textwidth]{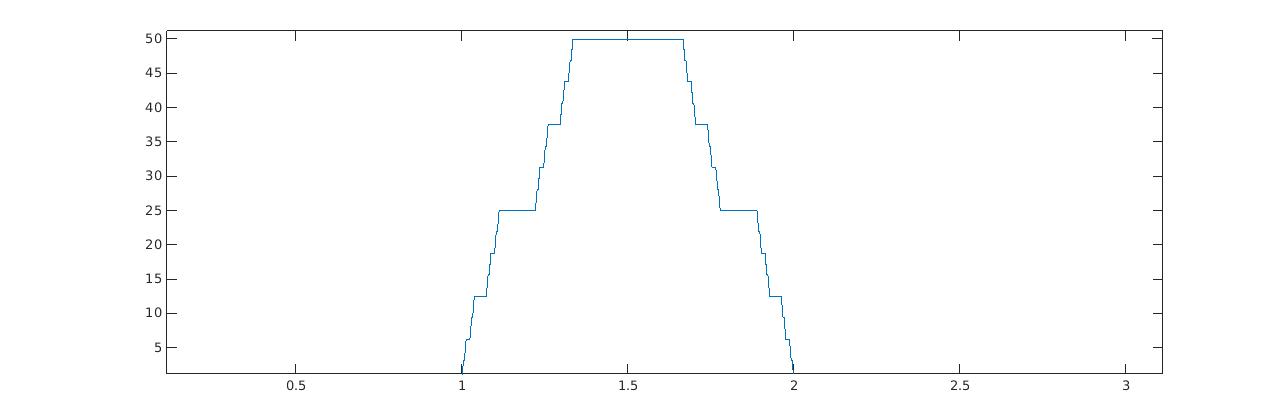}
	 \caption{In this figure we see one possible shape for $\overline{u}_i$.}
	 \end{figure}
Let us define $\overline{u}_i=\tilde{c}_i\cdot u_i(t)+\overline{c}_i$ with $\overline{c}_i\in \mathbb{R}$ and $\tilde{c}_i>0$. The distributional derivative of $\overline{u}_i$ has a positive part in $I_1$ and a negative part in $I_2$.
The measure $D_t\overline{u}_i$ is a Cantor measure with support in $PC$ where $D_t\overline{u}_i^+$ is supported in $I_1$ and $D_t\overline{u}_i^-$ in $I_2$. Following the instructions in section \ref{GeneralExample} "Construction of Test Examples" gives us an optimal control $\overrightarrow{u}$.
Similarly as above, one can construct functions $\tilde{p}_i$, for $i=1,\cdots,m$, such that each of them has finitely many plateaus with different signs.

In our numerical experiment we considered the following parameters:
\begin{itemize}
\item[a)] $\Omega=[-2,2]^2$, $T=5$, $m=1$,
\item[b)] $g(x):=10\cdot\mathbf{1}_{[-\frac{1}{2},\frac{1}{2}]^2}(x)$,
\item[c)] $\tilde{p}(t):=\varphi_{\tilde{\epsilon}}*(\mathbf{1}_{[\frac{1}{2},2]}+\mathbf{1}_{[3,4.5]})(t)$, with $\varphi_{\tilde{\epsilon}}(x):=c_{\tilde{\epsilon}}e^{\frac{-1}{1-
			\left(\tilde{\epsilon}^{-1}x\right)^2
			}}\cdot\mathbf{1}_{(-\tilde{\epsilon},\tilde{\epsilon})}(x)$,
$\tilde{\epsilon}=0.28$, $c_{\tilde{\epsilon}}:=\int_\mathbb{R}\varphi_{\tilde{\epsilon}}(x)dx$,

\item[d)]$f(x,y):=\varphi_1(x)\cdot\varphi_2(y)$ with $\varphi_i(x):=\exp\left(
-\frac{1}{1-x^2}
\right)\mathbf{1}_{[-1,1]}(x)\in C^\infty_c(\Omega)$, for $i=1,2$,
\item[e)] and the optimal control we want to approximate
\begin{align*}
\begin{matrix}
\overline{u}(t):=10\cdot C(\frac{t-0.8}{2(2.14-0.8)})\mathbf{1}_{[0.8,2.14]}(t) + 5\cdot \mathbf{1}_{[(2.14,2.85)]}(t)+10\cdot C(\frac{4.2-t}{2(4.2-2.85)})\cdot\mathbf{1}_{[2.85,4.2]}(t).
\end{matrix}
\end{align*}
\end{itemize}

In Figures \ref{CantorExample1} and \ref{CantorExample4} we depict the numerical optimal control for two different choices of d.o.f.
In the upper left sub-figure the \textcolor{red}{red} curve which is the approximated derivative of the approximated optimal control $\overline{u}_{\text{approx}}$. The blue curve represents an approximation to the derivative of $\overline{u}$ by finite differences.

The BV-path following algorithm starts with $\gamma_0=1$, $(v_0,c_0)=(0,0)$, and we iterate according to $\gamma_{k+1}=0.5\gamma_k$. We stopped the BV-path following algorithm when $\gamma_k=3.8\cdot 10^{-6}$ was reached. The function $\kappa$ is defined by $\kappa(\gamma)=0$.
We used $\|F_{\gamma_k}(u_k)\|_{L^2(I)^m\times\mathbb{R}^m}\leq 0.5\cdot 10^{-4}$ as the stopping criterion for the semi-smooth Newton algorithm. 

\begin{figure}
	\hspace*{-3.5cm}
	\includegraphics[width=1.5\textwidth]{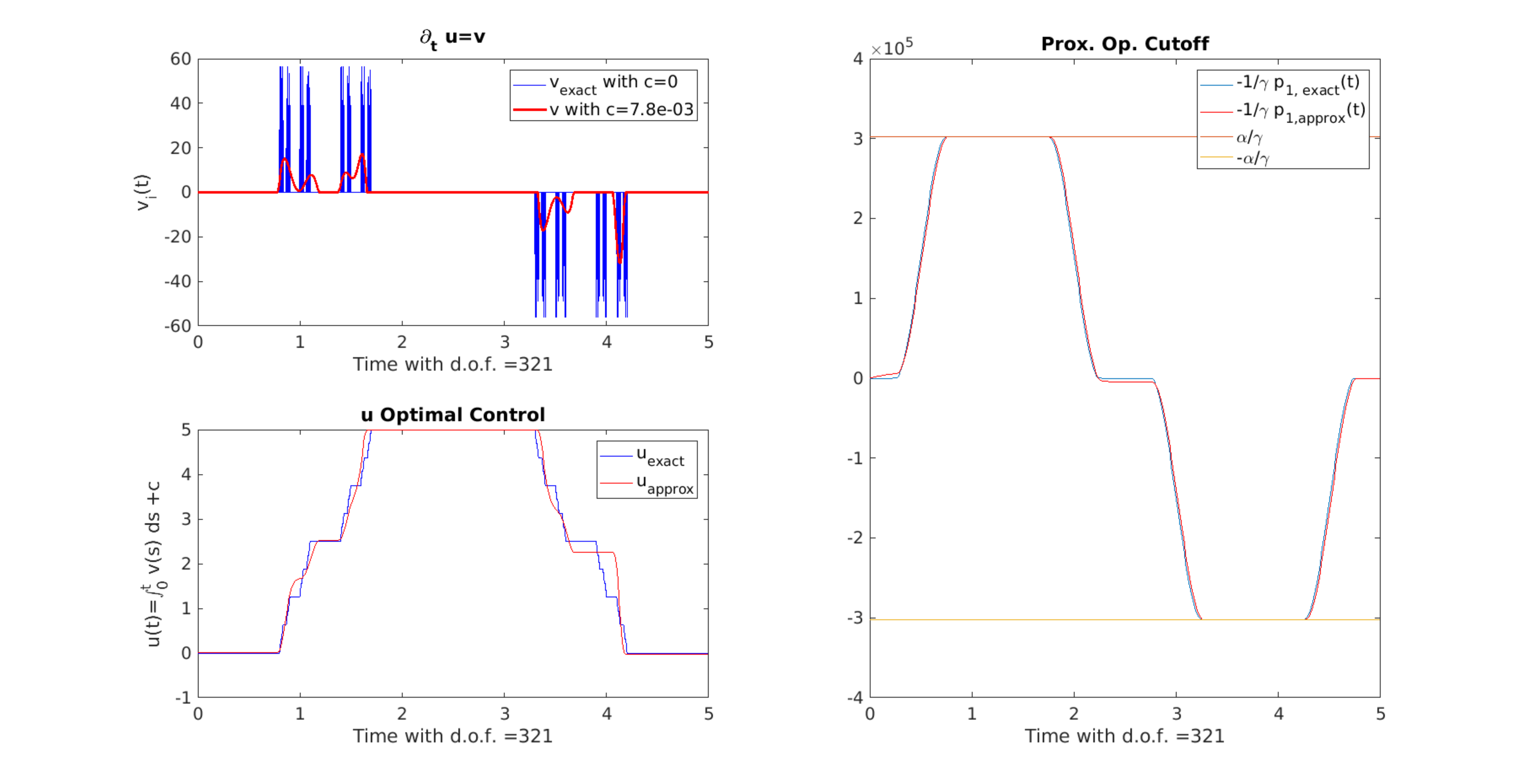}
	\caption{}
	\label{CantorExample1}
	\hspace*{-3.5cm}
	\includegraphics[width=1.5\textwidth]{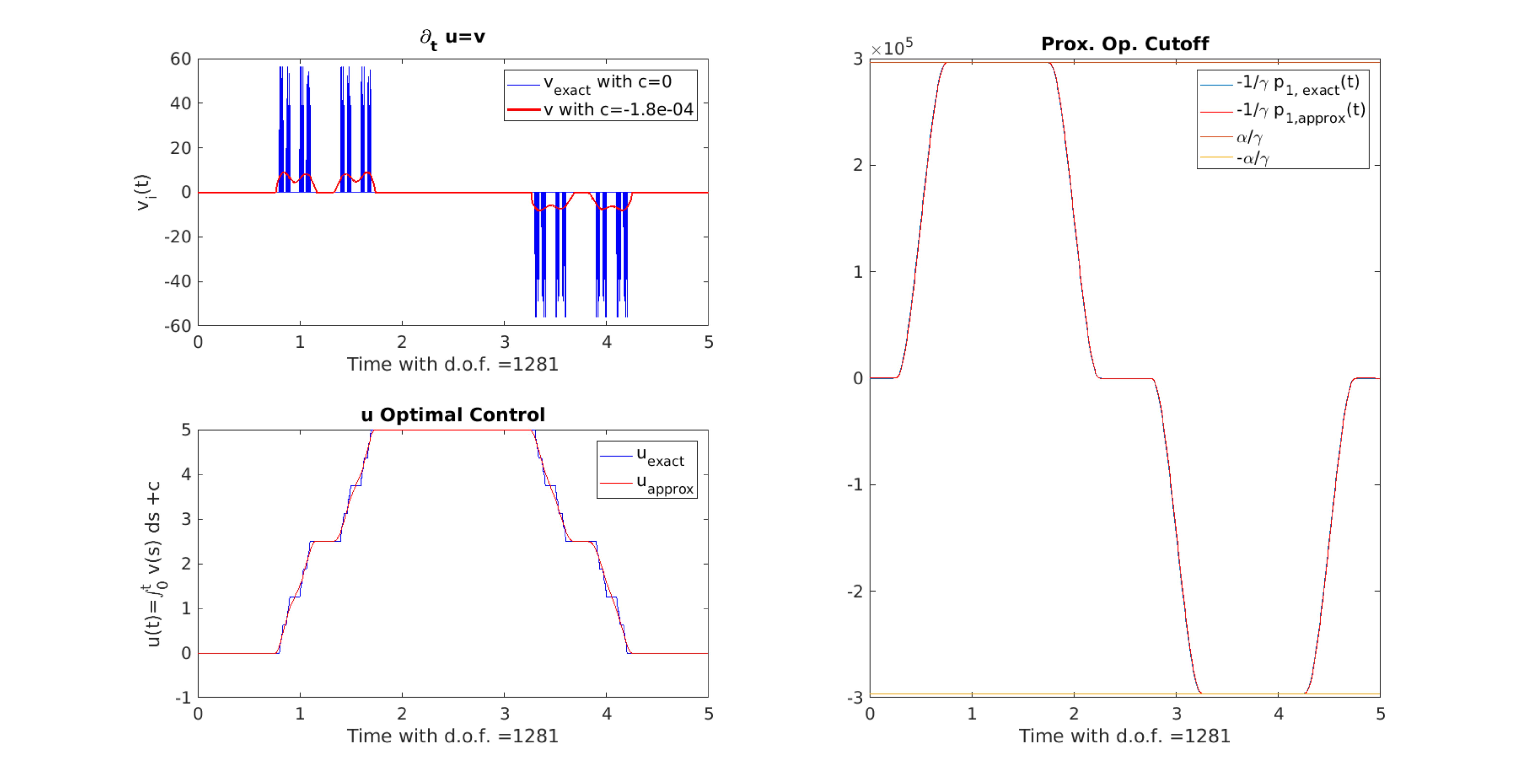}
	\caption{}
	\label{CantorExample4}
\end{figure}

\section{Remarks}
\label{remarkSect}
All results we present in this paper can be shown similarly for the following modifications of the homogeneous boundary conditions in $(\mathcal{W})$:
	\begin{align*}
	\begin{matrix}
	y &=& \phi_D  &\text{ on }&(0,T)\times \partial \Omega,\\
	a_N\frac{\partial y}{\partial \eta} +a_D y & = & \phi_N  &\text{ on }&(0,T)\times \partial \Omega, 
	\end{matrix}
	\end{align*}
	with $a_D\in \mathbb{R}$, and $a_N\neq 0$. The definition of the solution term in these cases can be found in \cite{lasiecka1986non}, \cite{LASIECKA1991112}.
	 More detailed discussion can be found in \cite{PhDEngel}.

\section{Appendix}
The following proof of Theorem \ref{FirstOrder} is adapted from the proof of \cite[Theorem 3.3]{[KCK]}.
\begin{proof} Let us define the linear continuous operators
\begin{align*}
\begin{matrix}	
\begin{matrix}
	P_i: & M(I)^m & \rightarrow & M(I)\\ \\
	& v & \mapsto & v_i
	\end{matrix} &, &
	\begin{matrix}
	\mathbb{D}: & M(I)^m \times \mathbb{R}^m & \rightarrow & M(I)^m
		\\ \\
		& (v,c) & \mapsto & v
	\end{matrix}
	\end{matrix}
\end{align*}
for $i=1,\cdots,m$. By the convexity of $(\tilde{P})$ we have that $\overrightarrow{u}=(\overrightarrow{v},\overrightarrow{c})\in M(I)^m\times\mathbb{R}^m$ is an optimal control of $(\tilde{P})$ if
	\begin{align}
	0\in \partial \left(
	\frac{1}{2}\| S(\overrightarrow{u})-y_d \|_{L^2(\Omega_T)}^2 + \sum\limits_{j=1}^m \alpha_j \|v_j\|_{M(I)}
	\right)\subseteq (M(I)^m\times \mathbb{R}^m)^*.
	\end{align}
	Defining $F(\overrightarrow{u}):=\frac{1}{2}\| S(\overrightarrow{u})-y_d \|_{L^2(\Omega_T)}^2$ we have for $0<\tilde{\tau}<1$, and $u=(v,c)\in M(I)^m\times \mathbb{R}^m$:
	\begin{align}\label{CCKTrick}
	\begin{matrix}
	0 \leq \frac{J(\overrightarrow{u}+\tilde{\tau} (u-\overrightarrow{u}))-J(\overrightarrow{u})}{\tilde{\tau}} = \frac{F(\overrightarrow{u}+\tilde{\tau} (u-\overrightarrow{u}))-F(\overrightarrow{u})}{\tilde{\tau}}+ \sum\limits_{j=1}^m \alpha_j \frac{\|P_j(\overrightarrow{v}+\tilde{\tau} (v-\overrightarrow{v}))\|_{M(I)}-\|P_j(\overrightarrow{v})\|_{M(I)}}{\tilde{\tau}}\\ \\
	\underbrace{\leq}_{\text{convexity}}  \frac{F(\overrightarrow{u}+\tilde{\tau} (u-\overrightarrow{u}))-F(\overrightarrow{u})}{\tilde{\tau}} + \sum\limits_{j=1}^m \alpha_j\frac{(1-\tilde{\tau})\|P_j(\overrightarrow{v})\|_{M(I)}+\tilde{\tau}\|P_j(v)\|_{M(I)}- \|P_j(\overrightarrow{v})\|_{M(I)}}{\tilde{\tau}}\\ \\
	\xrightarrow{\tilde{\tau} \rightarrow 0} DF_{\overrightarrow{u}}(u-\overrightarrow{u}) +\sum\limits_{j=1}^m \alpha_j\|P_j \circ \mathbb{D}(u)\|_{M(I)}-\sum\limits_{j=1}^m \alpha_j\|P_j \circ \mathbb{D}(\overrightarrow{u})\|_{M(I)}
	\end{matrix}
	\end{align}
	with $DF_{\overrightarrow{u}}$ the Gateaux derivative of $F$ in $\overrightarrow{u}$. It has the following form:
	\begin{align}
	\begin{matrix}
	DF_{\overrightarrow{u}}:=\begin{pmatrix}
	p_1(s)\\ p_1(0)
	\end{pmatrix}:= \begin{pmatrix}
	\int\limits_{s}^T\int\limits_\Omega L^* \left( S\begin{pmatrix}
	\overrightarrow{u}
	\end{pmatrix}-y_d \right) \overrightarrow{g}dxdt \\ \\
	\int\limits_{\Omega_T} L^* \left( S\begin{pmatrix}
	\overrightarrow{u}
	\end{pmatrix}-y_d \right) \overrightarrow{g}dxdt
	\end{pmatrix}
	\end{matrix}
	\end{align}
	Hence, (\ref{CCKTrick}) implies that
	\begin{align}\label{CCKTrick2}
	0\in
	DF_{\overrightarrow{u}} + \partial\left(\sum\limits_{j=1}^m \alpha_j \|P_j(\mathbb{D}(\overrightarrow{u}))\|_{M(I)}
	\right).
	\end{align}
	Using standard techniques for the sub differential of a convex functional and (\ref{CCKTrick2}), this implies that:
	\begin{align}
	\begin{matrix}
	-\begin{pmatrix}
	p_1(s)\\ p_1(0)
	\end{pmatrix}\in \begin{pmatrix}
	\left( \alpha_i \partial \| v_i \|_{M(I)}\right)_{i=1}^m \\
	0_{\mathbb{R}^m}.
	\end{pmatrix}
	\end{matrix}
	\end{align}
	which is equivalent to the following: For all
$i=1,...,m$ and $v\in M(I)$ it holds that\\
$\langle v-v_i,-p_{1,i}\rangle_{M(I),C_0(I)} \leqslant \alpha_i \|v\|_{M(I)} - \alpha_i \|v_i\|_{M(I)}$ and we have $p_1(0)=0_{\mathbb{R}^m}$. 
\end{proof}
Note that the regularity of $L^*(S(\overrightarrow{u})-y_d)\in C(\overline{I};L^2(\Omega))$, $p_1(0)=0_{\mathbb{R}^m}$, and the definition of $p_1$ imply that $p_1\in C_0(I)$ and thus $D_{\overrightarrow{u}}F(u-\overrightarrow{u})$ is an element of $(M(I)^m\times \mathbb{R}^m)^*$.

	In the following we prove Lemma \ref{maxminFormula1}:
	\begin{proof}
	Firstly, let us present the optimality conditions of $(\tilde{P}_\gamma)$:
	The control $(\overrightarrow{v},\overrightarrow{c})\in L^2(I)^m \times \mathbb{R}^m$ is optimal for $(\tilde{P}_\gamma)$ if
	\label{First_Version}
	\begin{align}\label{key32}
	\begin{matrix}
	\left( \begin{matrix} -\int\limits_{\Omega}\int\limits_s^T L^*(S(\overrightarrow{v},\overrightarrow{c})-y_d) \overrightarrow{g} dt dx -\gamma \overrightarrow{v}  \\ \\ -\int\limits_{\Omega_T} L^*(S(\overrightarrow{v},\overrightarrow{c})-y_d) \overrightarrow{g} dt dx -\kappa(\gamma)\overrightarrow{c}\end{matrix} \right) \in
	\left(
	\begin{matrix}
	(\alpha_i \partial \left( \| \overrightarrow{v}_i \|_{L^1(I)}\right)_{i=1}^m\\  \\ 0_{\mathbb{R}^m}
	\end{matrix}
	\right).
	\end{matrix}
	\end{align}
	Consider now $\overrightarrow{u}:=(\overrightarrow{v},\overrightarrow{c})\in L^2(I)^m \times \mathbb{R}^m$ and the following function
		\begin{align}\label{E1Append}
		\begin{matrix}
		\prox^\gamma_{\sum\limits_i\alpha_i \|\cdot\|_{L^1(I)}}(\overrightarrow{\mathfrak{p}})= \argmin\limits_{v\in L^2(I)^m}\left( \sum\limits_{i=1}^m \alpha_i \|v_i\|_{L^1(I)} + \frac{\gamma}{2}\sum_{i=1}^m \|v_i - p_i\|_{L^2(I)}^2 \right)\text{ for $\overrightarrow{\mathfrak{p}}\in L^2(I)^m$}
		\end{matrix}
		\end{align}
		which we also call the Prox problem. Our aim is to calculate the first-order optimality conditions for this problem, and to find an explicit representation for $\overrightarrow{v}(\overrightarrow{\mathfrak{p}}):=\prox^\gamma_{\sum\limits_i\alpha_i \|\cdot\|_{L^1(I)}}(\overrightarrow{\mathfrak{p}})$. For the sub-differential of the non-smooth term we obtain\\
$
		\partial \left( \sum\limits_{i=1}^m \alpha_i \|P_i(\cdot)\|_{L^1(I)^m}\right)=\sum\limits_{i=1}^m\alpha_i P_i^*\partial\| P_i(\cdot) \|_{L^1(I)} \subseteq L^2(I)^m
$
		with the domain $L^2(I)$ for the function $\alpha_i \|\cdot\|_{L^1(I)}$. Thus, we have the following first-order optimality conditions for $\prox^\gamma_{\sum\limits_i\alpha_i \|\cdot\|_{L^1(I)}}(\overrightarrow{\mathfrak{p}})$:
$
		\gamma(\overrightarrow{\mathfrak{p}}-\overrightarrow{v})\in (\alpha_i \partial \| P_i(\overrightarrow{v}) \|_{L^1(I)})_{i=1}^m.
$
		For (\ref{E1Append}) this means that $\overrightarrow{v}\in L^2(I)^m$ is optimal if and only if
$
		\frac{\gamma}{\alpha_i}(\mathfrak{p}_i - v_i) \in \partial \|v_i\|_{L^1(I)}\text{ for all }i=1,...,m,
$ and $\overrightarrow{\mathfrak{p}}\in L^2(I)^m$.

		Next, let us show that our optimal control can be explicitly written as
		\begin{align}
		\begin{matrix}
		\overrightarrow{v}=\left( \max(0,\mathfrak{p}_i-\frac{\alpha_i}{\gamma})+\min(0,\mathfrak{p}_i+\frac{\alpha_i}{\gamma}) \right)_{i=1}^m \subset L^2(I)^m
		\end{matrix}\label{maxminf}
		\end{align}
		We can proceed coordinate wise. By the definition of the sub-differential we have the equivalent inequality condition
$
		\left\langle \frac{\gamma}{\alpha}(\mathfrak{p}-\overline{v}),v-\overline{v}\right\rangle_{L^2} \leq \|v\|_{L^1}-\|\overline{v}\|_ {L^1}\text{ for all }v\in L^2(I).
$
 By a standard Lebesgue point argument $\overline{v}$ is an optimal control if and only if
		\begin{align*}
		\begin{matrix}
		(v-\overline{v})\frac{\gamma}{\alpha}(\mathfrak{p}-\overline{v})\leq |v|-|\overline{v}| \text{ holds for all }v\in L^2(I) \text{ a.e. in }I
		\end{matrix}
		\end{align*}
This can also be expressed as
		\begin{align*}
		\overline{v}(x)= \left\lbrace\begin{matrix}
		0 & , & \vert \mathfrak{p}(x) \vert < \frac{\alpha}{\gamma}\\
		\mathfrak{p}(x) + \frac{\alpha}{\gamma} & , & \mathfrak{p}(x) \leq -\frac{\alpha}{\gamma}\\
		\mathfrak{p}(x) - \frac{\alpha}{\gamma} & , & \mathfrak{p}(x) \geq \frac{\alpha}{\gamma}
		\end{matrix}\right.
		\end{align*}
		or equivalently as
$
		\overline{v}(x)=\max\left(0,\mathfrak{p}(x)-\frac{\alpha}{\gamma}\right)+\min\left(0,\mathfrak{p}(x)+\frac{\alpha}{\gamma}\right).
$
		Now, we can state the following equivalent first-order optimality conditions for the Prox-problem: For $\overrightarrow{\mathfrak{p}}\in L^2(I)^m$, $\overrightarrow{v}\in L^2(I)^m$ is optimal if and only if
		\begin{align*}
		\begin{matrix}
		\frac{\gamma}{\alpha_i}(\mathfrak{p}_i - v_i) \in \partial \|v_i\|_{L^1(I)}\text{ for all }i=1,...,m \\ \\
		\Leftrightarrow\\ \\
		\overrightarrow{v}=\left(\max(0,\mathfrak{p}_i-\frac{\alpha_i}{\gamma})+\min(0,\mathfrak{p}_i+\frac{\alpha_i}{\gamma})\right)_{i=1}^m\in L^2(I)^m.
		\end{matrix}
		\end{align*}
		Recall that $\overrightarrow{u}=(\overrightarrow{v},\overrightarrow{c})\in L^2(I)^m\times \mathbb{R}^m$ is the optimal control for $(\tilde{P}_\gamma)$ if and only if for all $i=1,...,m$ we have
		\begin{align}
		\begin{matrix}\label{(1)_help}
		-\gamma v_i - \int\limits_{\Omega}\int\limits_s^T L^*(S\overrightarrow{u}-y_d)g_i dtdx \in \alpha_i \partial \|v_i\|_{L^1(I)}
		\end{matrix}\\
		\begin{matrix}\label{(2)_help}
	 	-\int_{\Omega_T}L^*(S\overrightarrow{u}-y_d)\overrightarrow{g} dtdx-\kappa(\gamma) \overrightarrow{c}=0_{\mathbb{R}^m}.
		\end{matrix}
		\end{align}
		Returning to $(\tilde{P}_\gamma)$, define $\overrightarrow{\mathfrak{p}}$ as
		\begin{align*}
		\begin{matrix}
		\overrightarrow{\mathfrak{p}}:=\overrightarrow{v}-\frac{1}{\gamma}\left(\int\limits_{\Omega}\int\limits_s^T L^*(S(\overrightarrow{u})-y_d) \overrightarrow{g} dt dx +\gamma \overrightarrow{v}
		\right)=-\frac{1}{\gamma}\int\limits_{\Omega}\int\limits_s^T L^*(S(\overrightarrow{u})-y_d) \overrightarrow{g} dt dx
		\end{matrix}
		\end{align*}
		which implies the equation:
		\begin{align*}
		\begin{matrix}
		\gamma(\overrightarrow{\mathfrak{p}}-\overrightarrow{v})=-\gamma\overrightarrow{v} -\int\limits_{\Omega}\int\limits_s^T L^*(S(\overrightarrow{u})-y_d) \overrightarrow{g} dt dx  =
		\begin{Bmatrix}
		\text{ first line of the left hand side in (\ref{key32}).}
		\end{Bmatrix}		
		\end{matrix}
		\end{align*}
		Thus, the first-order optimality system of the Prox problem is equal to the first line of the first-order optimality system of $(\tilde{P}_\gamma)$ in (\ref{key32}). Note that we have for $\overrightarrow{v}$
		\begin{align*}
		\begin{matrix}
		\overrightarrow{v}=\prox^\gamma_{\sum\limits_i\alpha_i \|\cdot\|_{L^1(I)}}(\overrightarrow{\mathfrak{p}})=\prox^\gamma_{\sum\limits_i\alpha_i \|\cdot\|_{L^1(I)}}\left(-\frac{1}{\gamma}\int\limits_{\Omega}\int\limits_s^T L^*(S(\overrightarrow{u})-y_d) \overrightarrow{g} dt dx\right).
		\end{matrix}
		\end{align*}
		Now, we can use the min-max formula in (\ref{maxminf}) for $\prox^\gamma_{\sum\limits_i\alpha_i \|\cdot\|_{L^1(I)}}(\overrightarrow{\mathfrak{p}})$ to rewrite (\ref{(1)_help}). This implies, the equivalent first-order optimality system we had to prove.
	
	\end{proof}

\section*{Acknowledgments}

We thank Philip Trautmann for the support he gave by his helpful hints and discussions.

All authors gratefully acknowledge support from the International Research Training Group IGDK 1754 "Optimization and Numerical Analysis for Partial Differential Equations with Nonsmooth Structures", funded by the Austrian Science Fund (FWF) and the German Research Foundation (DFG): [W 1244-N18]. The second author also acknowledges partial support from the ERC advanced grant 668998 (OCLOC) under the EU H2020 research program.

\FloatBarrier

\end{document}